\documentclass[sn-mathphys,Numbered,iicol]{sn-jnl}% 

\hypersetup{ breaklinks=true, % splits links across lines 
}

\usepackage{graphicx}%
\usepackage{multirow}%
\usepackage{amsmath,amssymb,amsfonts}%
\usepackage{amsthm}%
\usepackage{mathrsfs}%
\usepackage[title]{appendix}%
\usepackage{xcolor}%
\usepackage{textcomp}%
\usepackage{manyfoot}%
\usepackage{booktabs}%
\usepackage{algorithm}%
\usepackage{algorithmicx}%
\usepackage{algpseudocode}%
\usepackage{listings}%
\theoremstyle{thmstyleone}%
\newtheorem{theorem}{Theorem}%  meant for continuous numbers

\newtheorem{claim}{Claim}
\newtheorem{lemma}{Lemma}
\newtheorem{corollary}{Corollary}

\theoremstyle{thmstyletwo}%
\newtheorem{remark}{Remark}%

\theoremstyle{thmstylethree}%
\newtheorem{definition}{Definition}%

\usepackage{times}
\usepackage{graphicx}
\usepackage{amsmath}
\usepackage{amssymb}
\usepackage{booktabs}
\usepackage{times}

\usepackage{enumitem}
\usepackage{graphicx}
\usepackage{url}
\usepackage{amssymb}
\usepackage{amsmath}
\usepackage{times}
\usepackage{bm}

\usepackage{subcaption}

\DeclareMathOperator{\arctanh}{arctanh}

\DeclareMathOperator{\sgn}{sgn}

\DeclareMathOperator{\rank}{rank}

\DeclareMathOperator{\drind}{Ind}

\newcommand{\Ind}[1]{\mathbb{I}\left\{{#1}\right\}}
\newcommand{\RR}{\mathbb{R}}
\newcommand{\CC}{\mathbb{C}}
\newcommand{\ZZ}{\mathbb{Z}}
\newcommand{\NN}{\mathbb{N}}

\newcommand{\inv}[1]{{#1}^\ddagger}
\newcommand{\invinv}[1]{{#1}^{\ddagger\ddagger}}
\newcommand{\finv}[1]{{#1}^\dagger}
\newcommand{\drinv}[1]{{#1}^\text{D}}
\newcommand{\drlinv}[1]{{#1}^{\text{LD}}}
\newcommand{\drinvdrinv}[1]{{#1}^{\text{DD}}}

{\qed}
\newcommand{\ignore}[1]{}

\usepackage{xcolor}
\usepackage{lipsum} % dummy text 
 
\begin{document}

\title{Generalized Inversion of Nonlinear Operators}

\author*[1]{\fnm{Eyal} \sur{Gofer}}\email{eyal.gofer@ee.technion.ac.il}

\author[1]{\fnm{Guy} \sur{Gilboa}}\email{guy.gilboa@ee.technion.ac.il}

\affil[1]{\orgdiv{Faculty of Electrical and Computer Engineering}, \orgname{Technion – Israel Institute of Technology}, \orgaddress{\city{Haifa}, \postcode{3200003}, \country{Israel}}}

\abstract{
Inversion of operators is a fundamental concept in data
processing. Inversion of linear operators is well studied, supported by established theory. When an inverse either does not exist or is not unique, generalized inverses are used.
Most notable is the Moore-Penrose inverse, widely used in physics, statistics, and various fields of engineering. This work investigates generalized inversion of nonlinear operators.

We first address broadly the desired properties of generalized inverses, guided by the Moore-Penrose axioms.
We define the notion for general sets, and then a refinement, termed \emph{pseudo-inverse}, for normed spaces. 
We present conditions for existence and uniqueness of a pseudo-inverse and establish theoretical results investigating its properties, such as continuity, its value for operator compositions and projection operators, and others. Analytic expressions are given for the pseudo-inverse of some well-known,  non-invertible, nonlinear operators, such as hard- or soft-thresholding and ReLU. %
We analyze a neural layer and discuss relations to wavelet thresholding.

Next, the Drazin inverse, and a relaxation, are investigated for operators with equal domain and range. We present scenarios where inversion is expressible as a linear combination of forward applications of the operator. Such scenarios arise for classes of nonlinear operators with vanishing polynomials, similar to the minimal or characteristic polynomials for matrices. Inversion using forward applications may facilitate the development of new efficient algorithms for approximating generalized inversion of complex nonlinear operators.
}

\ignore{}


\maketitle

\section{Introduction}\label{sec:intro}
Operator inversion is a fundamental mathematical problem concerning various fields in science and engineering. The vast majority of research devoted to this topic is concerned with generalized inversion of \emph{linear operators}. However, \emph{nonlinear operators} are extensively used today in numerous domains, and specifically in machine learning and image processing.
Given that many nonlinear operators are not invertible, it is highly instrumental to formulate the generalized inverse of nonlinear operators and to analyze its properties.

Generalized inversion of linear operators has been studied extensively since the 1950's, following the paper of \citet{penrose_1955}. As noted in \cite{baksalary2021moore,ben2003generalized}, that work rediscovered, simplified, and made more accessible the definitions first made by \citet{moore1920reciprocal}, in what is referred to today as \emph{the Moore-Penrose inverse}, often also called the matrix pseudo-inverse. Essentially, a concise and unique definition is given for a generalized inverse of any matrix, including singular square matrices and rectangular ones. A common use of the Moore-Penrose (MP) inverse is for solving linear least-square problems.
As elaborated in the work of \citet{baksalary2021moore}, celebrating 100 years for its discovery, the MP inverse has broadened the understanding of various physical phenomena, statistical methods, and algorithmic and engineering techniques.
We hope this work can facilitate the use of nonlinear generalized inversion in the context of data science.

The generalized inversion of linear operators in normed spaces was analyzed by \citet{nashed1974unified,nashed1976unified}, where the relation to projections is made explicit. Additional properties were provided by \citet{wang2003metric,ma2014perturbations,cui2016some}, and the references therein. A detailed overview of various generalized inverse definitions and properties is available in the book of \citet{ben2003generalized}.
The representation of Drazin inverses of operators on a Hilbert space is examined in \citet{du2005representation}. See 
\citet{ilic2017algebraic,wang2018generalized}
for recent books on the algebraic properties of generalized inverses, idempotent and projection operators on Banach spaces, generalized Drazin invertibility, and computational aspects of Moore-Penrose and Drazin inversion.  
Iterative numerical algorithms for approximating the generalized inverse were proposed, for example, in \citep{stickel1987class,climent2001geometrical,pan2018efficient}.

The research related to generalized inversion of nonlinear operators has been very scarce. In \citep{zervakis1992iterative}  the notion of \emph{nonlinear pseudo-inverse} is given, for the first time, to the best of our knowledge. It is in the context of least square estimation in image restoration. This topic, however, is not further developed. Characteristics of the pseudo-inverse and issues such as existence and uniqueness are not discussed. 
The most comprehensive study on generalized inversion of nonlinear operators is in the work of \citet{dermanis1998generalized}. 
While focused on geodesy, he defines inversion broadly. % 
We significantly extend the initial results of Dermanis and attempt to establish a general theory, relevant to data science. 
In control applications (see \citep{mickle2004unstable,liu2007singular}), pseudo-inversion is used for the design of controllers for nonlinear dynamics. The inversion is meant to approximately cancel the dynamic of the plant. In these studies, inversion is discussed in a rather narrow applied context.  

The goal of this paper is to define broad notions of generalized inversion of nonlinear operators. We present several flavors of inverses and attempt to analyze their mathematical properties in various settings.
We use the following nomenclature: all types of inverses that hold some inversion properties are referred to as \emph{generalized inverses.} 
We refer to \emph{pseudo-inverse} specifically for the case of the Moore-Penrose inverse in the linear case and its direct extension in the nonlinear case. 
Pseudo-inverse is one type of generalized inverse. Others which will be discussed in more detail are $\{1,2\}$-inverse, Drazin inverse, and left-Drazin inverse.

The main contributions of the paper are: 
\begin{enumerate}
\item
We explain and illustrate the general concept of generalized inversion of nonlinear operators for sets, as well as its fundamental properties. It relies on the first two axioms of Moore and Penrose and generally is not unique. 
\item
For nonlinear operators in normed spaces a stronger definition can be formulated, 
based on best approximate solution, %
which directly coincides with Moore-Penrose in the linear setting. 
We show that although the first Moore-Penrose axiom is implied, the second one is still required explicitly, unlike the linear case.

\item 
Certain theoretical results are established, such as conditions for existence and uniqueness, the domain and possible continuity of the inverse of a continuous operator over a compact set, and various settings in which the inverse of an operator composed by several simpler operators can be inferred.
\item
    Analytic expressions of the pseudo-inverse for some canonical functions are given, as well as explicit computations of a neural-layer inversion and relations to wavelet thresholding.
\item We then focus on endofunctions ($T:V\rightarrow V$) and investigate the Drazin inverse and a relaxation thereof. For both general sets and vector spaces, we show scenarios where these inverses are expressible using forward applications of the operator. In particular, this generalizes the approach of using the Cayley-Hamilton theorem for expressing the inverse of a matrix.
    
\ignore{}
    \end{enumerate}
    
\ignore{
}

{\bf Some definitions and notations.}
For an operator $T:V\rightarrow W$ between sets $V$ and $W$, $T(V')$ denotes the image of $V'\subseteq V$ by $T$, and $T^{-1}(W')$ denotes the preimage of $W'\subseteq W$ by $T$. The restriction of $T$ to a subset $V'\subseteq V$ is denoted by $T|_{V'}$. The composition of operators $T_1:V_2\rightarrow V_3$ and $T_2:V_1\rightarrow V_2$ is denoted by $T_1\circ T_2$  or $T_1 T_2$. The set of all operators from $V$ to $W$ is denoted by $W^V$, and we write $|V|$ for the cardinality of the set $V$. An operator $T:V\rightarrow W$ is \textit{idempotent} or a \textit{generalized projection} iff $T\circ T = T$. We write $\arg\min_{v\in V}\{f(v)\}$ or $\arg\min\{f(v): v\in V\}$ for the set of elements in $V$ that minimize a function $f:V\rightarrow \RR$. The closed ball with center $a$ and radius $r$ is denoted by $\bar{B}(a,r)$, and the open ball by $B(a,r)$.
The notation $\RR^+$ specifies non-negative reals.
The indicator function of an event $A$ is denoted by $\Ind{A}$. The determinant of a matrix $A$ is denoted by $\det(A)$ and its rank by $\rank{A}$. We write $F[x]$ for the ring of univariate polynomials over a field $F$. The degree of a polynomial $p$ is denoted by $deg(p)$, and is undefined for the zero polynomial.

 \section{The Moore-Penrose Properties and Partial Notions of Inversion}\label{sec:semi inverse}

The Moore-Penrose inverse is defined in the linear domain as follows. 
\begin{definition}[Moore-Penrose pseudo-inverse]
\label{def:pseudo-inv}
Let $V$ and $W$ be finite-dimensional inner-product spaces over $\CC$. Let $T:V\rightarrow W$ be a linear operator with the adjoint $T^*:W\rightarrow V$. The pseudo-inverse of $T$ is a linear operator $T^\dagger:W\rightarrow V$, which admits the following identities: 
\begin{align*}
&\textup{MP1: } TT^\dagger T = T &
&\textup{MP3: } (TT^\dagger)^* = TT^\dagger \\
&\textup{MP2: } T^\dagger T T^\dagger = T^\dagger &
&\textup{MP4: } (T^\dagger T)^* = T^\dagger T
\end{align*}

\end{definition} 
 As was shown by \citet{penrose_1955}, this pseudo-inverse exists uniquely for every linear operator. It may be calculated, for example, using singular value decomposition (SVD, see, e.g., \citep{ben2003generalized}). The same work by Penrose showed that the pseudo-inverse is involutive, namely, $T^{\dagger\dagger}=T$. He also showed \citep{penrose1956best} that it yields a \textit{best approximate solution} (BAS) to the equation $Tv=w$. That is, for every $w\in W$, $\|Tv-w\|_2$ is minimized over $v\in V$ by $v^*=T^\dagger(w)$, and among all such minimizers, it uniquely has the smallest $L_2$ norm.

Let us now examine how the Moore-Penrose scheme can be adapted to nonlinear operators. A direct application, unfortunately, does not work. Recall that the adjoint of a linear operator satisfies $\langle Tv,w \rangle = \langle v,T^*w \rangle$ for every $v\in V$, $w\in W$. As can be easily shown, the properties of the inner product restrict $T$ to be linear (see Claim~\ref{cla:adjoint means linear} in the appendix). This means that the adjoint operation as defined here does not extend to nonlinear operators, and that MP3--4, which involve the adjoint, need to be replaced. Thus, MP1--2 are extended in a way that is suitable for nonlinear operators. 

It should be noted that any possible subset of the four MP properties (and indeed other possible properties) may serve as the basis for the definition of a different type of inverse. These various inverses have been studied intensively for linear operators (see \citet{ben2003generalized}). In keeping with the notation of \citep{ben2003generalized}, the MP pseudo-inverse is referred to as a $\{1,2,3,4\}$-inverse. For any subset $\{i,j,\ldots, k\}\subseteq\{1,2,3,4\}$ one may also talk of $T\{i,j,\ldots, k\}$, the set of all $\{i,j,\ldots, k\}$-inverses of the operator $T$. The same notations will be used here for a nonlinear operator $T$ as well. Let us first look closer at $\{1,2\}$-inverses of nonlinear operators.

\section{The \{1,2\}-Inverses of Nonlinear Operators}\label{subsec:(1,2)}
It is important to note that $V$ and $W$ are no longer required to be inner-product or even vector spaces, and they may in fact be any two general nonempty sets. The following lemma pinpoints the nature of a $\{1,2\}$-inverse. %

\begin{lemma}\label{lem: equivalent meaning of MP1-2 different spaces}
Let $T:V\rightarrow W$ and $\inv{T}:W\rightarrow V$, where $V$ and $W$ are nonempty sets. The statement $\inv{T}\in T\{1,2\}$ is equivalent to the following: $\forall w\in T(V)$, $\inv{T}(w)\in T^{-1}(\{w\})$, and $\forall w\notin T(V)$, $\inv{T}(w)=\inv{T}(w')$ for some $w'\in T(V)$.
\end{lemma}
\begin{proof}
  Assume that $\inv{T}\in T\{1,2\}$. If $w\in T(V)$, then $w=T(v)$ for some $v\in V$. By MP1,
  $w=T(v)=T\inv{T}T(v)=T\inv{T}(w)$, 
  so $\inv{T}(w)\in T^{-1}(\{w\})$. If $w\notin T(V)$, let $w'=T(\inv{T}(w))$, so $w'\in T(V)$, and by MP2, $\inv{T}(w') =\inv{T}(w)$.

  In the other direction, for every $v$ we have $T(v)\in T(V)$, so $\inv{T}(T(v))\in T^{-1}(\{T(v)\})$, implying that $T\inv{T} T(v)=T(v)$, satisfying MP1. As for MP2, if $w\in T(V)$, we have $\inv{T}(w)\in T^{-1}(\{w\})$, so $T\inv{T}(w)=w$, and thus
  $\inv{T} T\inv{T}(w)=\inv{T}(w)$. If $w \notin T(V)$, then $\inv{T}(w)=\inv{T}(w')$ for some $w'\in T(V)$, and we already know that $\inv{T} T\inv{T}(w')=\inv{T}(w')$. As a result, $\inv{T} T\inv{T}(w)=\inv{T}(w)$, which completes the proof.
\end{proof}

Lemma~\ref{lem: equivalent meaning of MP1-2 different spaces} provides a recipe for constructing a $\{1,2\}$-inverse $\inv{T}$ of an operator $T$. First, for each element $w$ in the image of $T$, define $\inv{T}(w)$ as some arbitrary element in its preimage $T^{-1}(\{w\})$.
Second, for any element $w$ not in the image of $T$, select some arbitrary element in $T(V)$, and use its (already defined) inverse as the value for $\inv{T}(w)$. 

Thus, MP1--2 leave us with two degrees of freedom in defining the inverse. One is in selecting a subset $V_0$ of $V$, which contains exactly one source of each element in $T(V)$. It is easy to see that this set is exactly $\inv{T}T(V)$. The other is an arbitrary mapping $P_0$ from $W\setminus T(V)$ to $T(V)$, which may be extended to a mapping from $W$ to $T(V)$ by defining it as the identity mapping on $T(V)$. This mapping clearly equals $T\inv{T}$. The following theorem summarizes the resulting picture.

\begin{theorem}\label{thm:properties of new inverse}
Let $T:V\rightarrow W$ be an operator, let $V_0\subseteq V$ contain exactly one source for each element in $T(V)$, and let $P_0:W\rightarrow T(V)$ satisfy that its restriction to $T(V)$ is the identity mapping. Then the following hold.
\begin{enumerate}

    \item 
    The restriction $T|_{V_0}$ is a bijection from $V_0$ onto $T(V)$.
    \item 
    The function $\inv{T} = (T|_{V_0})^{-1} P_0$ is a $\{1,2\}$-inverse of $T$ that uniquely satisfies the combined requirements MP1--2, $T\inv{T}=P_0$, and $\inv{T}T(V)=V_0$.
    \item 
    Applying the construction of part~2 to $\inv{T}$ with the set $W_0 = T(V)$ and mapping $Q_0 = \inv{T} T$, yields $\invinv{T} = (\inv{T}|_{W_0})^{-1}Q_0=T$.
    \item 
The constructions of parts 2 and 3 generalize the MP pseudo-inverse for linear operators. Specifically, 
if $V=\CC^n$, $W=\CC^m$, and $T$ is linear, then picking $V_0 = T^\dagger T(V)$ and $P_0 = TT^\dagger$ yields $\inv{T}=(T|_{V_0})^{-1} P_0 = T^\dagger$ and picking $W_0 = T^\dagger(V)$ and $Q_0 = T^\dagger T$ yields 
$\invinv{T} = (\inv{T}|_{W_0})^{-1}Q_0=T^{\dagger\dagger}=T$.
    \item
    The functions $T\inv{T}$ and $\inv{T} T$ are idempotent.
   \item 
   If $T$ is a bijection of $V$ onto $W$, then $\inv{T} = T^{-1}$ is the only $\{1,2\}$-inverse of $T$.

\end{enumerate}
\end{theorem}
\begin{proof}
\begin{enumerate}
    \item 
Immediate, since $V_0$ contains exactly one source for each element in $T(V)$.
\item
    If $w\in T(V)$ then $P_0$ maps it to itself, and $(T|_{V_0})^{-1}$ then maps it to one of its sources. If $w\notin T(V)$ then $(T|_{V_0})^{-1} P_0$ maps it to the inverse of an element in $T(V)$. By Lemma~\ref{lem: equivalent meaning of MP1-2 different spaces}, $\inv{T}$ satisfies MP1--2. We have that $T\inv{T} = T(T|_{V_0})^{-1} P_0= P_0$ and 
    \begin{align}
    \inv{T} T(V) &= (T|_{V_0})^{-1} P_0 T(V) \\
    & = (T|_{V_0})^{-1}T (V)\\
    & = V_0\;.
    \end{align}
As for uniqueness, Lemma~\ref{lem: equivalent meaning of MP1-2 different spaces} describes the degrees of freedom in defining a $\{1,2\}$-inverse. For $w\in W\setminus T(V)$, $T\inv{T}(w)$ is clearly the element in $T(V)$ whose inverse is associated with $w$. Otherwise, $w\in T(V)$ and $\inv{T} T(V)=V_0$. Since $V_0$ contains exactly one source for $w$, there is no choice in defining its inverse.
\item
First we need to verify that $W_0\subseteq W$ contains exactly one source under $\inv{T}$ for each element in $\inv{T}(W)$. We have by parts 1 and 2 that $\inv{T}(W) = V_0$ and $W_0=T(V)$ has exactly one source under $\inv{T}$ for each element in $V_0$. We also have that $Q_0=\inv{T} T$ maps $V$ to $\inv{T}(W)$ and its restriction to $\inv{T}(W)$ is the identity mapping. By part~2 we thus have that defining $\invinv{T} = (\inv{T}|_{W_0})^{-1}Q_0$ complies with MP1--2 and satisfies $\inv{T}\invinv{T}=Q_0$ and $\invinv{T}\inv{T}(W)=W_0$. To show that $\invinv{T} = T$ we can use the uniqueness of the inverse shown in part~2. It is clear that $T\in\inv{T}\{1,2\}$ by Lemma~\ref{lem: equivalent meaning of MP1-2 different spaces}. In addition, $\inv{T} T=Q_0$ by definition, and $T\inv{T}(W) = T(V)=W_0$, and we are done.
\item
Let $V=\CC^n$, $W=\CC^m$, let $T(v)=Av$ for a complex $m$ by $n$ matrix $A$, and let $A^\dagger$ be its MP inverse. We set $V_0 = A^\dagger A (V)$ and $P_0 = AA^\dagger$. Thus, $P_0:W\rightarrow A(V)$, and for every $w=Av$, 
\begin{equation}
P_0(w) = AA^\dagger Av = Av = w\;,
\end{equation}
as required. In addition, $A(V_0) = AA^\dagger A(V) = A(V)$, so every element in $A(V)$ has at least one source in $V_0$. To show that there is exactly one source, it is enough to show that $V_0$ and $A(V)$, which are both vector spaces, have the same dimension. Let $A = U_1\Sigma U_2^*$ be the singular value decomposition of $A$, where $U_1$ and $U_2$ are unitary and $\Sigma$ is a generalized diagonal matrix. Then the MP inverse of $A$ is given by $A^\dagger = U_2\Sigma^\dagger U_1^*$, where $\Sigma^\dagger_{ji} = \Sigma^{-1}_{ij}\Ind{\Sigma_{ij}\neq 0}$. Thus $AA^\dagger = U_1 D_1 U_1^*$ and $A^\dagger A = U_2 D_2 U_2^*$, where $D_1 = \Sigma \Sigma^\dagger$ and $D_2 =  \Sigma^\dagger\Sigma$ are diagonal matrices with only ones and zeros, and the same trace. As a result, both $P_0(W)$ (which equals $A(V)$) and $V_0$ have an equal dimension that is the common trace of $D_1$ and $D_2$, as we wanted to show. 

The fact that this construction in fact yields $A^\dagger$ follows since 
\begin{equation}
A^\dagger = A^\dagger A A^\dagger = AP_0 = (A|_{V_0})^{-1}P_0\;,
\end{equation}
where the last equality holds since $A$ is a bijection from $V_0$ to $A(V)=P_0(W)$. As for $A^{\dagger\dagger}$ (part~3), both our inverse and the MP inverse yield $A$ (see \citep{penrose_1955}, Lemma~1), so they coincide again.
\item
Directly from MP1--2, we have that $T\inv{T} T\inv{T} = T\inv{T}$ and $\inv{T} T \inv{T} T = \inv{T} T$.
\item 
By Lemma~\ref{lem: equivalent meaning of MP1-2 different spaces}, there are no degrees of freedom in defining $\inv{T}$, since $W\setminus T(V)$ is empty and there is a single source for every element. On the other hand, it is clear that $T^{-1}$ satisfies MP1--2.%
\end{enumerate}
\end{proof}

It is instructive to consider the symmetry, or lack thereof, between $T$ and $\inv{T}$. The roles of $T$ and $\inv{T}$ are completely symmetric in MP1--2, and it is therefore possible to define $\invinv{T}=T$ if only MP1--2 have to be satisfied. The construction given in part~2 of Theorem~\ref{thm:properties of new inverse} does not restrict the degrees of freedom allowed by MP1--2, yet embodies them (through $V_0$ and $P_0$) in a way that is not necessarily symmetric in the roles of $T$ and $\inv{T}$. The specific construction for $\invinv{T}$ in part~3 aims to preserve the symmetry between $T$ and $\inv{T}$ and achieves the desired involution property of the inverse, namely, $\invinv{T} = T$. 

Ultimately, Theorem~\ref{thm:properties of new inverse} describes $T$ as a composition of an endofunction of $V$, $\inv{T} T$, and a bijection, $T|_{V_0}$, where $V_0 = \inv{T} T(V)$. Symmetrically, $\inv{T}$ is a composition of an endofunction of $W$, $T\inv{T}$, and a bijection $\inv{T}|_{W_0}$, which is the inverse of the bijection that is part of $T$.
These compositions, $T=T(\inv{T} T)$ and $\inv{T}=\inv{T}(T\inv{T})$, are inherent in MP1--2. The endofunctions above are also idempotent, or generalized projections. It should be noted that such operators do not require a metric, like conventional metric projections. The full scheme is depicted in Figure \ref{fig:symmetric scheme}.
\begin{figure*}[t]
\centering
\includegraphics[scale=0.6,trim={4cm 6cm, 5cm, 6cm},clip]{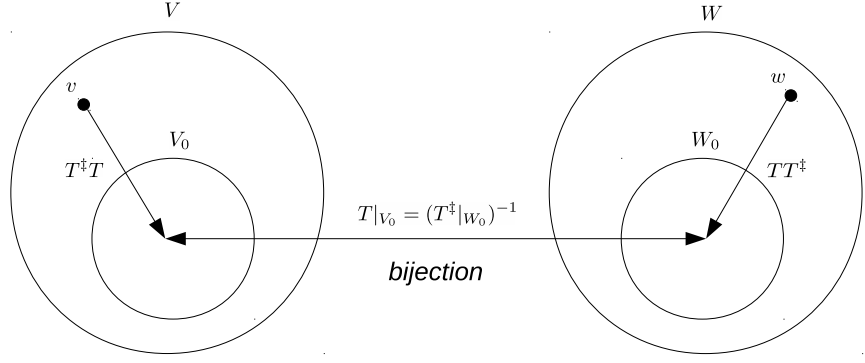} 
\caption{A symmetric depiction of the $\{1,2\}$-inverse scheme, as detailed in Theorem~\ref{thm:properties of new inverse}. The inverse `projects' a point $w$ to $W_0=T(V)$ and then maps it to $V_0$ using $T|_{V_0}^{-1}$ (equivalently, $\inv{T}|_{W_0}$), which is a proper bijection from $W_0$ onto $V_0$. In the other direction, $T$ equivalently first `projects' a point $v$ onto $V_0$, where $T$ is then a proper bijection onto $W_0$. The generalized projections are done by $\inv{T} T$ in $V$ and by  $T\inv{T}$ in $W$.}
\label{fig:symmetric scheme}
\end{figure*}

If one seeks to assign a particular $\{1,2\}$-inverse for every operator from $V$ to $W$ and indeed also from $W$ to $V$, then that is always possible by following the recipe of part~2 of Theorem~\ref{thm:properties of new inverse} or Lemma~\ref{lem: equivalent meaning of MP1-2 different spaces}. In addition, if this assignment is a bijection $F:W^V\rightarrow V^W$, then the involution property may be satisfied for all operators by defining the inverse of $F(T)$ for $T:V\rightarrow W$ as $T$. Note that this definition is proper since $F$ is a bijection, and that $\inv{T}\in T\{1,2\}$ implies $T\in\inv{T}\{1,2\}$ since the roles of the operator and the inverse in MP1--2 are symmetrical.

The existence of a bijection between $W^V$ and $V^W$ (equivalently, $|W^V|=|V^W|$) is a necessary condition for all operators to have an involutive $\{1,2\}$-inverse. For example, if $|V^W|<|W^V|$, then there are two different operators $T_1,T_2\in W^V$ that are assigned the same inverse in $\inv{T}\in V^W$, and clearly the inverse of $\inv{T}$ cannot be equal to both.

\subsection{The \{1,2\}-Inverse of Generalized Projections}\label{subsec:generalized projections}
Every idempotent endofunction $E:V\rightarrow V$ has that $E^3=E$. Consequently, $E\in E\{1,2\}$. This may not hold for the MP inverse. For example, the matrix 
\begin{equation}
E= \begin{pmatrix} 0 & 0\\1 & 1 \end{pmatrix}\;    
\end{equation}
is idempotent, yet its MP inverse is
\begin{equation}
E^\dagger = \begin{pmatrix} 0 & 0.5\\0 & 0.5 \end{pmatrix}\;.
\end{equation}

There are classes of idempotent operators where the MP inverse is necessarily the operator itself. For a linear $T:\CC^n\rightarrow\CC^m$, $TT^\dagger$ and $T^\dagger T$ are not merely idempotent endomorphisms. SVD yields that both $TT^\dagger$ and $T^\dagger T$ have the self-adjoint matrix form $UDU^*$, where $U$ is unitary and $D$ is square diagonal with ones and zeros on the diagonal. It is easy to see that both $TT^\dagger$ and $T^\dagger T$ are their own MP inverses by checking MP1--4 directly. The form $UDU^*$ describes a metric projection onto a vector subspace, which is in particular a closed convex subset of the original vector space. 
For considering metric projections in more general scenarios, the following concept will be key.
\begin{definition}
A nonempty subset $C$ of a metric space $V$, with a metric $d$, is a Chebyshev set, if for every $v\in V$, there is exactly one element $v'\in C$ s.t. $d(v,v') = \inf_{u\in C} d(v,u)$. The projection operator onto $C$,  $P_C:V\rightarrow V$, may thus be uniquely defined at $v$ by setting $P_C(v) = v'$.\footnote{This definition of course applies to normed spaces and Hilbert spaces.}
\end{definition}
By the Hilbert projection theorem (see, equivalently, \citep{rudin87book}, Theorem~4.10), every nonempty, closed, and convex set in a Hilbert space is a Chebyshev set. 
\begin{claim}\label{cla:inverse of projection}
Let $V$ be a metric space, $C\subseteq V$ a Chebyshev set, and $P_C:V\rightarrow V$ the projection onto $C$.
\begin{enumerate}
    \item The operator $P_C$ is idempotent.
    \item It holds that $P_C$ is its own $\{1,2\}$-inverse, but if $|C|>1$ and $C\subsetneq V$, it is not unique.
    \item In the particular case where $V$ is a Hilbert space and $C$ is a closed subspace of $V$, it holds that $P_C$ is a linear operator and a $\{1,2,3,4\}$-inverse of itself.
\end{enumerate}
\end{claim}
\begin{proof}
\begin{enumerate}
    \item 
    For every $v\in V$ we have $P_C(v)\in C$, and for every $v\in C$, $P_C(v)=v$, therefore, for every $v\in V$, $P_C^2(v)=P_C(v)$.
    \item
    Since $P_C$ is idempotent, clearly $P_C\in P_C\{1,2\}$. By Lemma~\ref{lem: equivalent meaning of MP1-2 different spaces}, we may define a $\{1,2\}$-inverse $\inv{P}_C$ as $\inv{P}_C(v)=v$ for every $v\in C$, and $\inv{P}_C(v)=\inv{P}_C(T(v))$ for every $v\notin C$, for any arbitrary $T:V\setminus C\rightarrow C$. Thus, $\inv{P}_C(v)=T(v)$ for every $v\notin C$. Since $|C|>1$ and $C\subsetneq V$, we may choose $T$ that does not coincide with $P_C$.
    \item
    Note that any subspace is clearly convex, so any closed subspace is a Chebyshev set by the Hilbert projection theorem. The fact that the projection onto a closed subspace is linear is well known (see, e.g., \citep{rudin87book}, Theorem~4.11). To show that $P_C^\dagger=P_C$, it remains to show that MP3--4 are satisfied. Since $P_C^2=P_C$, both MP3 and MP4 would be satisfied if $P_C^*=P_C$, or equivalently, if for every $u,v\in V$, $\langle P_C(u),v \rangle = \langle u,P_C(v) \rangle$. This indeed holds, since 
    \begin{align}
    \langle P_C(u),v \rangle &= \langle P_C(u),P_C(v)+P_{C^\perp}(v) \rangle\\
    &= \langle P_C(u),P_C(v) \rangle \\
    &=  \langle P_C(u)+P_{C^\perp}(u),P_C(v) \rangle\\
    &=  \langle u,P_C(v) \rangle\;,
    \end{align}
    completing the proof.
\end{enumerate}
\end{proof}

Part~3, given here for completeness, relates the result back to the linear case, and for $V=\CC^n$ yields the (unique) MP inverse.

\section{A Pseudo-Inverse for Nonlinear Operators in Normed Spaces}\label{sec:axioms nonlinear}

We now examine the concept of pseudo-inverse, which applies to normed spaces and coincides with MP1--4 for matrices. For this reason we will use the notation $\finv{T}$ for this type of inverse as well. We note that this definition does not use the adjoint operation and can be applied to nonlinear operators.

\begin{definition}[Pseudo-inverse for normed spaces]
\label{def:generalized pseudo-inv}
Let $V$ and $W$ be subsets of normed spaces over $F$ ($\RR$ or $\CC$), and let $T:V\rightarrow W$ be an operator. A pseudo-inverse of $T$ is an operator $\finv{T}:W\rightarrow V$ that satisfies the following requirements:% 
\ignore{

}
\begin{align*}
  \textup{BAS: } &\forall w\in W, \finv{T}(w)\in \arg\min\{\|v\|:v\in\\ &\arg\min_{v'\in V}\{\|T(v')-w\|\}\,\}\\
  \textup{MP2: } & T^\dagger T T^\dagger = T^\dagger
\end{align*}
\end{definition} 
The calculation of $\finv{T}(w)$ for a given $w\in W$ can be translated into two consecutive minimization problems: first, find $m_w = \min_{v\in V}\{\|T(v)-w\|\}$, and then minimize $\|v\|$ for $v\in V$ s.t. $\|T(v)-w\|=m_w$. Note that the definition implicitly requires that the minima exist. Any solutions are then also required to satisfy MP2.

The BAS property is modeled after the best approximate solution property of the MP inverse. It replaces MP3 and MP4, so the definition no longer relies on the adjoint operation. In addition, BAS directly implies MP1, making it redundant. This means that a pseudo-inverse according to Definition~\ref{def:generalized pseudo-inv} is in particular a $\{1,2\}$-inverse. Property MP2 is not implied by BAS, and neither is the involution property.
\begin{claim}\label{cla:new def dependencies}
Let $V$ and $W$ be subsets of normed spaces, let $T:V\rightarrow W$, and let $S:W\rightarrow V$ satisfy BAS w.r.t. $T$.
\begin{enumerate}
    \item 
    It holds that $S\in T\{1\}$.
    \item
    For every $w\in T(V)$, $STS(w)=S(w)$, namely, MP2 is satisfied on $T(V)$.
    \item
    In general, $S$ does not necessarily satisfy MP2.
    \item
    It might be impossible to define involutive pseudo-inverses for all operators in $W^V$, even when all these operators have pseudo-inverses.
\end{enumerate}
\end{claim}
\begin{proof}
\begin{enumerate}
    \item 
    Let $v\in V$ and write $w = T(v)$. For $v'\in V$ to minimize $\|T(v')-w\|$, one must have $T(v')=w$. By the BAS property, $S(w)\in T^{-1}(\{w\})$, so $TS(w) = w$, and therefore, $TS T(v) = T(v)$, namely, MP1 is satisfied.
\item
Let $w=T(v)$ for some $v\in V$. Therefore, 
\begin{align}
STS(w) &= STST(v) = ST(v)\\
&=S(w)\;,    
\end{align} 
where the second equality is by part~1, so MP2 is satisfied on $T(V)$.
\item
Let $V=\{-1,1\}$ and $W = \{0,1\}$ as subsets of $\RR$, let $T(-1)=T(1)=1$, and consider $S(0)$. Since $T(V)=\{1\}$, any source of $1$ with minimal norm would satisfy BAS, namely, either $-1$ or $1$. The same is true for $S(1)$. Thus, BAS poses no restrictions on $S$. We can therefore satisfy BAS by choosing $S(0)=-1$ and $S(1)=1$, and then $STS(0)=1\neq -1  = S(0)$, contradicting MP2. 
\item
Let $V$ be a set of elements with equal norms, e.g., on the unit circle in $\RR^2$ with $L_2$.

Assume that $T$ is onto $W$. By parts 1 and 2, a pseudo-inverse of $T$ is in $T\{1,2\}$. BAS means that for every $w\in W$, a pseudo-inverse value at $w$ must be a source of $w$ with minimal norm, namely, any source of $w$. This requirement is fulfilled by any $\{1,2\}$-inverse of $T$ (Lemma~\ref{lem: equivalent meaning of MP1-2 different spaces}), so the pseudo-inverses of $T$ are exactly its $\{1,2\}$-inverses. 

Assume now that $T_1$ is a constant operator, $T_1(v)=w_0$ for every $v\in V$. BAS imposes no restrictions on the value of a pseudo-inverse of $T_1$ at any point, so again any $\{1,2\}$-inverse of $T_1$ from $W$ to $V$ would be a pseudo-inverse.

For every $T\in W^V$, it is always possible to pick a $\{1,2\}$-inverse by Lemma~\ref{lem: equivalent meaning of MP1-2 different spaces}. 

Let $|W|=2$ and $|V|=5$, so $|W|^{|V|}>|V|^{|W|}$, and every operator in $W^V$ is either surjective or constant. It is thus possible to define a pseudo-inverse for each, but there must be two distinct operators $T_1,T_2\in W^V$ with the same pseudo-inverse, so the involution property cannot possibly hold for both. 
\end{enumerate}
\end{proof}

The nonlinear pseudo-inverse has the desirable elementary property that for any bijection $T$, $T^{-1}$ is its unique valid pseudo-inverse (Claim~\ref{cla:full inv of a bijection} in the appendix). By itself, however, Definition~\ref{def:generalized pseudo-inv} implies neither existence nor uniqueness.  In what follows we will seek interesting scenarios where this is indeed the case. One such important scenario is the original case of matrices, discussed by Penrose. As shown by him \citep{penrose1956best}, for a linear operator $T$, the BAS property is uniquely satisfied by the MP inverse. Since the MP inverse also satisfies MP2, we can state the following.
\begin{claim}\label{cla:generalizing MP}
Let $V$ and $W$ be finite-dimensional inner-product spaces over $\CC$, and let $T:V\rightarrow W$ be a linear operator. Then Definition~\ref{def:generalized pseudo-inv} w.r.t. the induced norms is uniquely satisfied by the MP inverse.
\end{claim}

Focusing on the image of the operator (equivalently, if the operator is onto) allows for easier characterization of existence and uniqueness.

\begin{lemma}[Inverse restricted to the image]\label{lem:inverse on the image}
Let $V$ and $W$ be subsets of normed spaces, and let $T:V\rightarrow W$ and $E\subseteq T(V)$.
A pseudo-inverse $\finv{T}:E\rightarrow V$ exists iff $A_w = \arg\min\{\|v\|: v\in T^{-1}(\{w\})\}$ is well-defined (minimum exists) for every $w\in E$. A pseudo-inverse $\finv{T}:E\rightarrow V$ exists uniquely iff $A_w$ is a singleton for every $w\in E$.
\end{lemma}
\begin{proof}
If $A_w$ is well-defined for every $w\in E$, then BAS can be satisfied. By part~2 of Claim~\ref{cla:new def dependencies}, BAS implies MP2, and thus a pseudo-inverse exists. Moreover, if $A_w$ is a singleton for every $w$, then there are no degrees of freedom in choosing a pseudo-inverse, so it is unique.

If a pseudo-inverse exists, then the BAS property implies that for every $w\in E$, $A_w$ is well-defined. Suppose that for some $w$, $A_w$ is not a singleton. BAS then allows a degree of freedom for the value of a pseudo-inverse on $w$. Since BAS implies MP2 in the current scenario, any choice would yield a valid pseudo-inverse, implying non-uniqueness.
\end{proof}

With the following restriction, the pseudo-inverse may be uniquely defined beyond the image.
\begin{claim}[Sufficient condition for a unique pseudo-inverse]\label{cla:exist and unique}
Assume that for every $w\in T(V)$, $\arg\min\{\|v\|:T(v)=w\}$ is a singleton, and let $W' = \{w\in W: \arg\min\{\|w_0-w\|:w_0\in T(V)\}\,\textup{is a singleton}\}$. Then a pseudo-inverse $\finv{T}:W'\rightarrow V$ exists and it is unique.
\end{claim}
\begin{proof}
Note that $W'\supseteq T(V)$, and define $\finv{T}(w)$ as follows. If $w\in T(V)$, then $\finv{T}(w)$ is the single source of $w$ with minimal norm. If $w\in W'\setminus T(V)$, then define $\finv{T}(w)=\finv{T}(w_0)$, where $w_0$ is the single closest element to $w$ in $T(V)$. By Lemma \ref{lem: equivalent meaning of MP1-2 different spaces}, this is a $\{1,2\}$-inverse, and in addition, it clearly satisfies BAS, so it is a pseudo-inverse. The BAS property uniquely determines the values of a pseudo-inverse on $W'$, so we are done.
\end{proof}

We now consider the important case of continuous operators. The next theorem shows that a pseudo-inverse always exists over the whole range if the domain of the operator is compact. 
Conditions for uniqueness and continuity of the inverse are also given.

\begin{theorem}[Continuous operators over a compact set]\label{thm:inverse for continuous functions}
Let $V$ and $W$ be subsets of normed spaces where $V$ is compact, and let $T:V\rightarrow W$ be continuous. 
\begin{enumerate}
\item
A pseudo-inverse $\finv{T}:W\rightarrow V$ exists.
\item
Assume further that $W$ is a Hilbert space, $T$ is injective, and $T(V)$ is convex. Define $\finv{T}:W\rightarrow V$, as $\finv{T} = T^{-1}\circ P_{T(V)}$, where $P_{T(V)}$ is the projection onto $T(V)$ and $T^{-1}:T(V)\rightarrow V$ is the real inverse of $T$. Then $\finv{T}$ is continuous and is the unique pseudo-inverse of $T$.
\end{enumerate}
\end{theorem}
\begin{proof}

\begin{enumerate}
    \item 
    For every $w\in W$, the function $g_w:V\rightarrow \RR$ defined as $g_w(v) = \|T(v)-w\|$ is continuous. By the extreme value theorem, it attains a global minimum $m_w$ on $V$. The set $G_w = \{v\in V:\|T(v)-w\|=m_w\}=g_w^{-1}(\{m_w\})$ is closed, as the preimage of a closed set by a continuous function. As a subset of $V$, $G_w$ is also compact, and the function $h:G_w\rightarrow \RR$, defined as $h(v)=\|v\|$ is continuous and attains a minimum $\mu_w$ on $G_w$, again by the extreme value theorem. The set of elements in $G_w$ with minimal norm will be denoted by $H_w$. By definition, the BAS property can be satisfied by defining $\finv{T}(w)\in H_w$. However, this must be done while also satisfying MP2. By Lemma~\ref{lem: equivalent meaning of MP1-2 different spaces}, MP1--2 are satisfied if we choose $\finv{T}(w)\in T^{-1}(\{w\})$ for $w\in T(V)$, and $\finv{T}(w)=\finv{T}(w')$ for some $w'\in T(V)$ for $w\notin T(V)$. For $w\in T(V)$ we pick $\finv{T}(w)$ arbitrarily in $H_w$, and indeed $\finv{T}(w)\in H_w\subseteq G_w = T^{-1}(\{w\})$. For $w\notin T(V)$, picking $\finv{T}(w) = \finv{T} T(v)$ for some arbitrary $v\in H_w$ would thus guarantee that we satisfy MP1--2, and it remains only to show that $\finv{T} T(v)\in H_w$ if $v\in H_w$. Now, it holds that $H_w = \{u\in V:\|T(u)-w\|=m_w\;\textup{and}\;\|u\|=\mu_w\}$. By MP1 we have that $\|T(\finv{T} T(v))-w\|=\|T(v)-w\|=m_w$, so $\finv{T} T(v)\in G_w$ and it is left to show that $\|\finv{T} T(v)\|=\mu_w$. It holds that $\finv{T} T(v) $ is a source of $T(v)$ under $T$ with minimal norm. Since $v$ is also a source of $T(v)$, $\|\finv{T} T(v)\|\leq \|v\|=\mu_w$. However, the minimal norm of elements in $G_w$ is $\mu_w$, so $\|\finv{T} T(v)\|=\mu_w$, and we are done.
    \item
    Since $T$ is continuous and $V$ is compact, then $T(V)$ is compact, and therefore, closed. The projection $P_{T(V)}$ is thus well defined, by the Hilbert projection theorem, and it is continuous, by a theorem due to Cheney and Goldstein (Theorem~\ref{thm:projection is continuous} in the appendix). 

The operator $T$ is a bijection from $V$ to $T(V)$, so the real inverse $T^{-1}:T(V)\rightarrow V$ is well defined. Since $V$ is compact and $T$ is continuous, it is well known that $T^{-1}$ is also continuous. This holds since the preimage of any closed (and hence compact) set in $V$ by $T^{-1}$ is its image by $T$; this image is compact, since $T$ is continuous, and hence closed. 

Define $\finv{T}:W\rightarrow V$ as $\finv{T} = T^{-1}\circ P_{T(V)}$. This operator is continuous as the composition of two continuous operators. By Claim~\ref{cla:exist and unique} there is a unique pseudo-inverse of $T$ defined on $W$. To satisfy the BAS property, this pseudo-inverse must coincide with $\finv{T}$, and thus this unique pseudo-inverse must be $\finv{T}$.
\end{enumerate}

\end{proof}

One may show examples of continuous operators with a unique pseudo-inverse that nevertheless have a discontinuity even on the image. For $V=[-r,r]$ (with, say, $r\geq 4$) and $W=\RR$, the pseudo-inverse of $T(v)=(v-2)^3-(v-2)$ has a discontinuity since some points in the image have more than one source (Figure \ref{fig:example discontinuity}). The existence and uniqueness of a pseudo-inverse of $T$ on the image is guaranteed by Claim~\ref{cla:exist and unique}. The only violation of the assumptions of part~2 of Theorem~\ref{thm:inverse for continuous functions} is that $T$ is not injective.
\begin{figure*}[t]
\centering
\includegraphics[scale=0.8,trim={0cm 1.65cm, 0cm, 2.3cm},clip]{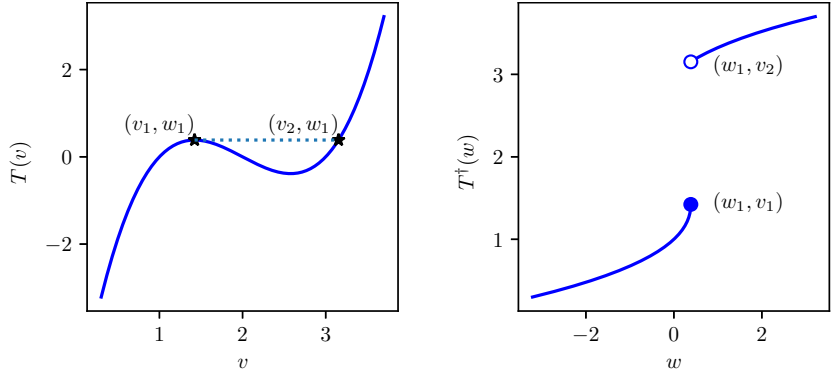} 
\caption{An example of a discontinuity in a unique pseudo-inverse due to local extrema, with $T(v)=(v-2)^3-(v-2)$. The pseudo-inverse value at $w_1=T(v_1)$ is $v_1$, which of the two sources of $w_1$ has the smaller norm. For any $\epsilon>0$, $T(v_1)+\epsilon$ has a single source that is greater than $v_2$.}
\label{fig:example discontinuity}
\end{figure*}

Another example of a discontinuity, this time also with $w\notin T(V)$, may be shown with $V=[0,2\pi]$, $W=\CC$, and $T(v)=e^{i v}$. A pseudo-inverse $\finv{T}:\CC\rightarrow [0,2\pi]$ exists by part~1 of Theorem~\ref{thm:inverse for continuous functions}. The image is $T(V)=\{z\in \CC: |z|=1\}$. For every $w\neq 0$, the nearest element in the image is $w/|w|$ and its source with minimal norm is its angle in $[0,2\pi)$. For $w=0$ all elements of the image are equidistant, and $v=0$ is the source with minimal norm. Thus 
 \begin{equation}
 \finv{T}(w) = 
 \begin{cases}
 \textup{angle of} \,w\, \textup{in} \,[0,2\pi) & w\neq 0\\
 0 & w=0
\end{cases}
\end{equation}
is the unique pseudo-inverse of $T$, with a discontinuity at any $w\in \RR^+$.

The impact of Theorem~\ref{thm:inverse for continuous functions} may be extended to non-compact domains $V$ that are a countable union of compact sets $V_1\subseteq V_2\subseteq\ldots$, for example, $\RR^n = \cup_{m=1}^\infty \bar{B}(0,m)$. For every point $w\in W$ one may relate the sequence $\finv{(T|_{V_n})}(w)$ to $\finv{T}(w)$, as follows.

\begin{theorem}\label{thm: limit of compact inverses}
Let $T:V\rightarrow W$, where $V=\cup_{n=1}^\infty V_n$ and $V_1\subseteq V_2 \subseteq \ldots$, denote $T_n = T|_{V_n}$, and assume a pseudo-inverse $\finv{T}_n:W\rightarrow V_n$ exists for every $n$. If for every $w\in W$ there is $N=N(w)\in \NN$ s.t. $n\geq N$ implies $\finv{T}_n(w) = \finv{T}_N(w)$, then $\finv{T}:W\rightarrow V$, defined as $\finv{T}(w) =  \lim_{n\rightarrow\infty} \finv{T}_n(w)$, is a pseudo-inverse for $T$. Furthermore, if the pseudo-inverse for each $T_n$ is unique, then $\lim_{n\rightarrow\infty} \finv{T}_n$ is the unique pseudo-inverse for $T$.
\end{theorem}
\begin{proof}
Let $w\in W$, and let $N_1=N(w)$, $N_2 = N(T(\finv{T}_{N_1}(w)))$, and $n\geq \max\{N_1,N_2\}$. Then
\begin{align}
    \finv{T}T\finv{T}(w) &=\finv{T}T\finv{T}_{N_1}(w) =\finv{T}_{N_2}T\finv{T}_{N_1}(w)\\ &=\finv{T}_{n}T\finv{T}_{n}(w)  
    =\finv{T}_{n}(w)
    =\finv{T}(w)\;,
\end{align}
so $\finv{T}$ satisfies MP2. We next show that $\finv{T}$ satisfies BAS. Assume by way of contradiction that $v_0 = \finv{T}_{N_1}(w)$ does not minimize $\|T(v)-w\|$ over $V$. Then there is some $v'\in V$ s.t. $\|T(v')-w\| < \|T(v_0)-w\|$. However, $v'\in V_n$ for some $n\geq N_1$, and $v_0=\finv{T}_n(w)$, and is therefore a minimizer of $\|T(v)-w\|$ over $V_n$, leading to a contradiction. Similarly, if there were a minimizer $v'$ of $\|T(v)-w\|$ over $V$ with a smaller norm than $v_0$, then $v'\in V_n$ for some $n\geq N_1$, and $v_0=\finv{T}_n(w)$ would again lead to a contradiction. Therefore, $v_0\in \arg\min\{\|v\|: v\in V\,\textup{minimizes}\, \|T(v)-w\|\}$, so $\finv{T}$ satisfies BAS, and is a valid pseudo-inverse for $T$.

For the uniqueness part, assume by way of contradiction that there is another pseudo-inverse $\bar{T}$ for $T$, which differs from $\finv{T} = \lim_{n\rightarrow\infty} \finv{T}_n$ on some $w_0$. Then there is some $n\geq N(w_0)$ s.t. both $\finv{T}(w_0)=\finv{T}_n(w_0)\in V_n$ and also $\bar{T}(w_0)\in V_n$. We will define a pseudo-inverse $\bar{T}_n\neq \finv{T}_n$ for $T_n$, yielding a contradiction to the uniqueness of $\finv{T}_n$. Denote $v=\finv{T}_n(w_0)$ and $\bar{v}=\bar{T}(w_0)$, and define  

\begin{equation}
\bar{T}_n(w) = \begin{cases}
      \bar{v}, & T_n\finv{T}_n(w) = T_n(\bar{v})\\
      &\quad\textup{or}\;w=w_0\\
      \finv{T}_n(w), & \textup{otherwise.}
    \end{cases}
    \end{equation}

\ignore{

}

First, $\bar{T}_n:W\rightarrow V_n$, and $\bar{T}_n\neq \finv{T}_n$ since they differ on $w_0$. We next show MP2 for $\bar{T}_n$. If $w=w_0$, then $\bar{T}_n T_n\bar{T}_n(w)=\bar{T}_n T_n(\bar{v})$. Now, $T_n\finv{T}_n(T_n(\bar{v}))=T_n(\bar{v})$, so $\bar{T}_n (T_n(\bar{v})) = \bar{v}$. Thus $\bar{T}_n T_n\bar{T}_n(w) = \bar{v} = \bar{T}_n(w)$, as required. If $T_n\finv{T}_n(w) = T_n(\bar{v})$, then $\bar{T}_n(w)=\bar{v}$ and $\bar{T}_n T_n\bar{T}_n(w)=\bar{T}_n T_n(\bar{v})=\bar{T}_n (T_n\finv{T}_n(w))$. By MP1, $T_n\finv{T}_n( T_n\finv{T}_n(w))=T_n\finv{T}_n(w) = T_n(\bar{v})$, so $\bar{T}_n (T_n\finv{T}_n(w))=\bar{v}$, and again MP2 is satisfied. Finally, assume $T_n\finv{T}_n(w) \neq T_n(\bar{v})$ and $w\neq w_0$, so $\bar{T}_n(w)= \finv{T}_n(w)$. Now, $\bar{T}_n T_n\bar{T}_n(w)=\bar{T}_n T_n \finv{T}_n(w)$, and $T_n \finv{T}_n(T_n \finv{T}_n(w)) = T_n \finv{T}_n(w)\neq T_n(\bar{v})$, so if we could show that $T_n \finv{T}_n(w)\neq w_0$, we would have that $\bar{T}_n T_n \finv{T}_n(w) = \finv{T}_n T_n \finv{T}_n(w) = \finv{T}_n(w)$, and thus MP2 for the final case. However, if $T_n \finv{T}_n(w)= w_0$, then $w_0\in T(V)$ and therefore $T(\bar{v})=w_0$, so $T_n(\bar{v})=T(\bar{v})=T_n \finv{T}_n(w)$, which contradicts the assumption for the current case.

The BAS property clearly holds for $\bar{T}_n$ with $w$ s.t. $\bar{T}_n(w) = \finv{T}_n(w)$, since $\finv{T}_n$ is a valid pseudo-inverse for $T_n$. For the remaining cases, we need to show that $\bar{v}\in \arg\min\{\|v\|: v\in V_n\,\textup{minimizes}\, \|T(v)-w\|\}$. If $w=w_0$, then $v,\bar{v}\in V_n$ and $v=\finv{T}_n(w_0)$ imply that $\|T_n(v)-w_0\|\leq \|T_n(\bar{v})-w_0\|$. On the other hand, $\bar{v}=\bar{T}(w_0)$ means that $\|T(\bar{v})-w_0\|\leq \|T(v)-w_0\|$, and thus $\|T(\bar{v})-w_0\| = \|T(v)-w_0\|$. Similarly, it now follows that $\|v\|=\|\bar{v}\|$, so $\bar{v}$ satisfies BAS as a choice for a pseudo-inverse of $T_n$ at $w_0$. If $T_n \finv{T}_n(w) = T_n(\bar{v})$, then $\|T_n \finv{T}_n(w) - w\|=\|T_n(\bar{v})-w\|$, and therefore $\bar{v}\in \arg\min_{u\in V_n}\{\|T_n(u)-w\|\}$. Since $\bar{T}(T(\bar{v})) = \bar{T}(T(\bar{T}(w_0))) = \bar{T}(w_0)= \bar{v}$, we have that
$\bar{T}(T_n \finv{T}_n(w)) = \bar{T}(T_n(\bar{v}))=\bar{v}$. This means that $\bar{v}$ has minimal norm among the sources of $T_n \finv{T}_n(w)$ under $T$, which include $\finv{T}_n(w)$, and thus $\|\bar{v}\|\leq\|\finv{T}_n(w)\|$. Therefore, $\bar{v}$ satisfies BAS as a pseudo-inverse for $T_n$ at $w$, and the proof is complete.
\end{proof}

\subsection{Projections Revisited}\label{subsec:proj revisited}
The $\{1,2\}$-inverse of projections has been considered in Subsection \ref{subsec:generalized projections}. We now continue the discussion in the current context of pseudo-inverses. The following result gives the pseudo-inverse of a projection for Hilbert spaces. This will be done as a special case of a more general scenario.
\begin{claim}[Pseudo-inverse of a projection cascade]\label{cla:projection cascade}
 Let $V$ be a subset of a Hilbert space, and let $1\leq n\in\NN$, $0\in C_n\subseteq C_{n-1} \dots \subseteq C_1 \subseteq V$, where $C_1,\ldots,C_n$ are closed and convex. Let $P_{C_i}:V\rightarrow V$ be the projection onto $C_i$ for $i=1,\ldots,n$, and define $P=P_{C_n}\circ\cdots\circ P_{C_1}$. Then $P_{C_n}$ is the unique pseudo-inverse of $P$.
\end{claim}
\begin{proof}
It holds that $P_{C_n}PP_{C_n}=P_{C_n}$, so $P_{C_n}$ satisfies MP2. Note that $P(V)=P(C_n)=C_n$, and for any $w\in V$, since $P_{C_n}(w)$ is uniquely the nearest vector to $w$ in $C_n$ ,
\begin{align}
A &= \arg\min_{v\in V}\{\|P(v)-w\|\}\\
&= \{v\in V: P(v)=P_{C_n}(w)\}\;.    
\end{align}
For the BAS property to hold, it remains to show that for any $v\in V$ s.t. $P(v)=P_{C_n}(w)$, $\|v\|>\|P_{C_n}(w)\|$ unless $v=P_{C_n}(w)$. If
$v\in C_n$ then $v = P(v) = P_{C_n}(w)$ and the claim is trivial, so assume $v\notin C_n$. Let then $1 \leq i\leq n$ be the smallest index s.t. $v \notin C_i$. Denote $u_j = P_{C_j}\circ\ldots\circ P_{C_1}(v)$ for any $1\leq j\leq n$ and $u_0=v$. For any $1\leq j \leq n$,
\begin{align}
\|u_{j-1}\|^2 =& \| u_{j-1} - u_j\|^2+\|u_j\|^2
\\&+2\textup{Re}\langle u_{j-1} - u_j, u_j\rangle\;.
\end{align}
\ignore{

}
By Lemma~\ref{lem:cheny-goldstein lemma}, using the fact that $0\in C_j$, we have that $\textup{Re}\langle 0-u_j, u_j-u_{j-1} \rangle\geq 0$, so $\textup{Re}\langle u_{j-1}-u_j,u_j \rangle\geq 0$, and therefore 
\begin{equation}
\|u_{j-1}\|^2 \geq \| u_{j-1} - u_j\|^2+\|u_j\|^2 \;.
\end{equation}
It follows that $\|u_0\| \geq \ldots \geq \|u_n\|$, and since $u_{i-1} \neq u_i$, we also have $\|u_{i-1}\| > \|u_i\|$, so $\|v\|=\|u_0\| > \|u_n\|=\|P(v)\|=\|P_{C_n}(w)\|$, as required.
\end{proof}
It should be noted that a projection cascade, as defined in Claim~\ref{cla:projection cascade}, is a generalized projection, but not necessarily a metric projection (for example, in the case of a square contained in a disk in $\RR^2$ with $L_2$). Note also that the requirement $0\in C$ for a convex set $C$ is equivalent to the requirement $P_C(0)=0$. Thus, the relation between this subclass of projections to all projections resembles that of linear operators to affine operators. The pseudo-inverse of a single projection may now be characterized easily.

\begin{corollary}\label{cor:projection general inverse}
Let $V$ be a subset of a Hilbert space, let $C\subseteq V$ be a closed and convex set s.t. $0\in C$, and let $P_C:V\rightarrow V$ be the projection onto $C$. Then $P_C$ is its own unique pseudo-inverse.
\end{corollary}

In the above results and in what follows projections will mostly be dealt with in Hilbert spaces, where nonempty closed and convex sets are known to be Chebyshev sets. Since pseudo-inverses are defined in the broader context of normed spaces, we 
will briefly 
mention a stronger characterization and refer the reader to \citep[p.~436]{megginson1998introduction} for more details.
\begin{theorem}
A normed space is rotund and reflexive iff each of its nonempty
closed convex subsets is a Chebyshev set.
\end{theorem}
\begin{remark}
Rotund and reflexive normed spaces include all uniformly rotund Banach spaces \citep[Proposition~5.2.6, Theorem~5.2.15]{megginson1998introduction}, and these, in turn, include $L_p(\Omega, \Sigma, \mu)$ spaces, for $1<p<\infty$, and every Hilbert space \citep{clarkson1936uniformly}.
\end{remark}

\section{High-Level Properties of the Nonlinear Inverse}\label{sec:high level properties}
It is interesting to consider general constructions, such as operator composition, product spaces and operators, domain restriction, and their relation to the inverse. The purpose is to describe the inverse for complex cases using simpler components. In this section we will consider both $\{1,2\}$-inverses and pseudo-inverses.

\begin{claim}[Product operator]\label{cla:product operator}
Let $T_i:V_i\rightarrow W_i$, where $V_i$ and $W_i$ are sets for $1\leq i \leq n$, and let $\inv{T}_i\in T_i\{1,2\}$. Define the product operator $\Pi_i T_i : \Pi_i V_i\rightarrow \Pi_i W_i$ as $(\Pi_i T_i)((v_1,\ldots,v_n)) = (T_1(v_1),\ldots, T_n(v_n))$. Then
\begin{enumerate}
    \item 
    The operator $\inv{(\Pi_i T_i)}:\Pi_i W_i\rightarrow \Pi_i V_i$, defined by 
\begin{align*}
\inv{(\Pi_i T_i)}((w_1,&\ldots,w_n)) = \\
&(\inv{T}_1(w_1),\ldots,
\inv{T}_n(w_n))\;,
\end{align*}
is a $\{1,2\}$-inverse of $\Pi_i T_i$. In addition, the involution property is preserved, namely, $\Pi_i T_i\in(\inv{\Pi_i T_i}) \{1,2\}$.
\item
For every $i$, assume further that $V_i$ and $W_i$ are subsets of normed spaces and $\inv{T}_i$ is also a pseudo-inverse of $T_i$. Consider $\Pi_i V_i$ and $\Pi_i W_i$ as subsets in the normed direct product spaces, each equipped with a norm that is some function of the norms of components, where that function is strictly increasing in each component (e.g., the sum of component norms). Then $\inv{\Pi_i T_i}$ is also a pseudo-inverse of $\Pi_i T_i$.
\end{enumerate}
\end{claim}
\begin{proof}
The proof of part~1 is immediate by considering each component separately. For part~2 we need only establish the BAS property for the product. Now, there is some $g:\RR^{+n}\rightarrow\RR^+$ s.t.
\begin{align}
\|&(\Pi_iT_i)(\Pi_i v_i) - \Pi_i w_i\|_{\Pi_i W_i}= g(\|T_1(v_1)\\&-
w_1\|_{W_1},\ldots, \|T_n(v_n)-w_n\|_{W_n})\;,
\end{align}
\ignore{

}
where the subscript of each norm specifies the appropriate space. Moreover, $g$ is strictly increasing in each component, so the l.h.s. is minimized iff each argument on the r.h.s. is minimized. Furthermore, 
\begin{align}
\|&(\inv{T}_1 (w_1),\ldots,\inv{T}_n( w_n))\|_{\Pi_iV_i}\\
&= f(\|\inv{T}_1(w_1)\|_{V_1},\ldots,\|\inv{T}_n(w_n)\|_{V_n})\;,
\end{align}
\ignore{

}
where $f:\RR^{+n}\rightarrow\RR^+$ is strictly increasing in each component, so $\Pi_i v_i = (\inv{T}_1 (w_1),\ldots,\inv{T}_n( w_n))$ is a minimizer of $\|(\Pi_iT_i)(\Pi_j v_j) - \Pi_i w_i\|_{\Pi_i W_i}$ that also has a minimal norm, as required.
\end{proof}

\begin{remark}
In the above claim, the norm of each product space may be the $L_p$ norm of the vector of component norms, for any $1\leq p < \infty$, which is strictly increasing in each component. In particular, if each $V_i$ and $W_i$ is $\RR$ with the absolute value norm, and $\sigma:\RR\rightarrow\RR$ has pseudo-inverse $\finv{\sigma}:\RR\rightarrow\RR$, then $T:\RR^n\rightarrow\RR^n$, the entrywise application of $\sigma$, has pseudo-inverse $(\finv{\sigma},\ldots,\finv{\sigma})$ for $L_p$, $1\leq p<\infty$.
\end{remark}
\begin{remark}
Note that the dependence on component norms mentioned in Claim~\ref{cla:product operator} may be different for the two direct product spaces. In addition, the proof holds even if this dependence is only non-decreasing in each component for the space containing $\Pi_i V_i$.
\end{remark}

\begin{definition}[Norm-monotone operators]
Let $V$, $W$ be subsets of normed spaces. An operator $T:V\rightarrow W$ is \textup{norm-monotone} if for every $v_1,v_2\in V$, $\|v_1\|\leq \|v_2\|$ implies $\|T(v_1)\|\leq \|T(v_2)\|$. 
\end{definition}

\begin{claim}[Composition with bijections]\label{cla:comp invertible}
Let $V$, $V_1$, $W$, and $W_1$ be sets, let $T:V\rightarrow W$, and let $S_1:W\rightarrow W_1$ and $S_2:V_1\rightarrow V$ be bijections. 
\begin{enumerate}
\item 
If $\inv{T}\in T\{1,2\}$, then $S_2^{-1}\inv{T} S_1^{-1}\in (S_1T S_2)\{1,2\}$.
\item
Assume further that $V$, $V_1$, $W$, and $W_1$ are subsets of normed spaces, $a S_1$ is an isometry for some $a\neq 0$, and $S_2^{-1}$ is norm-monotone. If $\inv{T}$ is a pseudo-inverse of $T$ then $S_2^{-1}\inv{T} S_1^{-1}$ is a pseudo-inverse of $S_1 T S_2$.
\end{enumerate}
\end{claim}
\begin{proof}
For part~1, we have that
\begin{align}
(S_1TS_2)(S_2^{-1}\inv{T} S_1^{-1})(S_1TS_2)&=S_1T\inv{T} TS_2 
\\&= S_1TS_2\;,
\end{align}
and 
\begin{align}
(S_2^{-1}&\inv{T} S_1^{-1})(S_1TS_2)(S_2^{-1}\inv{T} S_1^{-1})
\\&=S_2^{-1}\inv{T} T \inv{T} S_1^{-1}\\
&= S_2^{-1}\inv{T} S_1^{-1}\;,    
\end{align}
\ignore{

}
satisfying MP1--2. 

For part~2, by the BAS property of $\inv{T}$, it holds for every $w_1\in W_1$ that $\inv{T}(S_1^{-1}(w_1))\in \arg\min\{\|v\|: v\in A_{w_1}\}$, where $A_{w_1} = \{v\in V\,\textup{minimizes}\, \|T(v)-S_1^{-1}(w_1)\|\}$. We have that 
\begin{align}%
A_{w_1}&=\{v\in V\,\textup{minimizes}\, \|S_1T(v)-w_1\|\}\label{eq: bijection composition sets1}\\
\label{eq: bijection composition sets2}    
&=S_2(\{u\in V_1\,\textup{minimizes}\,\|S_1TS_2(u)\\
&\hspace{10pt}-w_1\|\})\;,
\end{align}
\ignore{

}
where the first equality holds since $aS_1$ is an isometry for some $a\neq 0$, and the second one holds since $S_2$ is bijective. Let $v'\in A_{w_1}$ have minimal norm, and consider $u'=S_2^{-1}(v')$. By Equations~\eqref{eq: bijection composition sets1}--\eqref{eq: bijection composition sets2}, $u'$ minimizes  $\|S_1TS_2(u)-w_1\|$ over $V_1$. Let $u''\in V_1$ be any minimizer of $\|S_1TS_2(u)-w_1\|$ in $V_1$, so $v''=S_2(u'')\in A_{w_1}$. Since $v'$ has minimal norm, we have that $\|S_2(u')\|\leq \|S_2(u'')\|$, and since $S_2^{-1}$ is norm-monotone, it holds that $\|u'\|\leq\|u''\|$. In other words, $u'$ has minimal norm among minimizers of $\|S_1TS_2(u)-w_1\|$ in $V_1$. We can pick $v' = \inv{T}(S_1^{-1}(w_1))$, so $u' =S_2^{-1}\inv{T}(S_1^{-1}(w_1))$, and thus 
\begin{align}\label{eq:inverse for s1Ts2}
    S_2^{-1}&\inv{T}(S_1^{-1}(w_1))\in\arg\min\{\|u\|: u\in V_1\,\\&\textup{minimizes}\, \|S_1T(S_2(u))-w_1\|\}\;.
\end{align}
\ignore{

}
That is, $S_2^{-1}\inv{T} S_1^{-1}$ satisfies the BAS property as required from a pseudo-inverse of $S_1 T S_2$.
\end{proof}

\begin{corollary}\label{cor:affine manipulations of operator}
If $V$ and $W$ are normed spaces, and $T:V\rightarrow W$ has a pseudo-inverse $\finv{T}:W\rightarrow V$, then for any scalars $a,b\neq 0$ and vector $w_0\in W$, $ b^{-1}\finv{T}(a^{-1}(w-w_0))$ is a pseudo-inverse of $aT(bv)+w_0$.
\end{corollary}
\begin{proof}
Let $V_1=V$, $W_1=W$, $S_1:W\rightarrow W_1$ defined as $S_1(w)=aw+w_0$, and $S_2:V_1\rightarrow V$ defined as $S_2(v)=bv$. Clearly, $S_1$ and $S_2$ are bijections that satisfy the conditions of part~2 of Claim~\ref{cla:comp invertible}. Then $S_2^{-1}\finv{T} S_1^{-1}=b^{-1}\finv{T}(a^{-1}(w-w_0))$ is a pseudo-inverse of $S_1TS_2$.
\end{proof}

\begin{remark}
Norm-monotone bijections form a fairly rich family of operators. Given normed spaces $V$ and $W$, for any bijection $T_1:\{v\in V:\|v\|=1\}\rightarrow \{w\in W:\|w\|=1\}$ and any surjective and strictly increasing $f:\RR^+\rightarrow\RR^+$, the operator $T:V\rightarrow W$, defined as  $T(v)=f(\|v\|)T_1(v/\|v\|)$ for $v\neq 0$ and $T(0)=0$, is a norm-monotone bijection. Note also that any bijective isometry from $V$ to $W$ that maps $0$ to $0$, multiplied by some nonzero scalar, is a norm-monotone bijection. Bijective isometries between normed spaces over $\RR$ are always affine by the Mazur-Ulam theorem \citep{mazur1932transformations}, but may be non-affine for spaces over $\CC$.
\end{remark}

\begin{claim}[Domain restriction changes the inverse]\label{cla:restriction}
Let $V,W$ be sets, $V_1\subsetneq V$, $T:V\rightarrow W$, and $\inv{T}\in T\{1,2\}$. Let $T_1 = T|_{V_1}$, and let $\inv{T}_1\in T_1\{1,2\}$. Then for every $w\in T(V)\setminus T(V_1)$, necessarily $\inv{T}_1(w)\neq \inv{T}(w)$.
\end{claim}
\begin{proof}
By Lemma~\ref{lem: equivalent meaning of MP1-2 different spaces}, for every $w\in T(V_1)$, $\inv{T}_1(w)$ must be a source of $w$ under $T$ in $V_1$ and $\inv{T}(w)$ must be a source of $w$ under $T$ in $V$, so it is possible (but not necessary) that $\inv{T}(w)=\inv{T}_1(w)$. However, for an element $w\in T(V)\setminus T(V_1)$, $\inv{T}(w)$ is a source of $w$ under $T$, while $\inv{T}_1(w)$ is a source of an element in $T(V_1)$. Therefore, $\inv{T}_1(w)\neq \inv{T}(w)$ necessarily. 
\end{proof}

\begin{remark}
If $\inv{T}_1$ and $\inv{T}$ are pseudo-inverses, such a necessary divergence between $\inv{T}$ and $\inv{T}_1$ might occur even for some $w\in T(V_1)$. Since $\inv{T}_1(w)$ must be a source in $V_1$ under $T$ with smallest norm, and $\inv{T}(w)$ must be a source under $T$ with the smallest norm in $V$, these smallest norms might be different, constraining that $\inv{T}_1(w)\neq\inv{T}(w)$.
\end{remark}

\begin{claim}[Composition]\label{cla:inv of composition}
Let $U,V,W$ be sets, and let $T:U\rightarrow V$ and $S:V\rightarrow W$.
\begin{enumerate}
    \item 
    If $\inv{S}_1\in S|_{T(U)}\{1,2\}$ and $\inv{T}\in T\{1,2\}$, then $\inv{T} \inv{S}_1 \in (ST)\{1,2\}$. Consequently, if $\inv{S}\in S\{1,2\}$ and $T$ is onto $V$, then $\inv{T}\inv{S}$ is a
    $\{1,2\}$-inverse of $ST$, but if $T$ is not onto, that does not necessarily hold.
    \item
    Assume that $U$, $V$, and $W$ are subsets of normed spaces, and let $\finv{T}$ and $\finv{S}_1$ be pseudo-inverses of $T$ and $S|_{T(U)}$, respectively. There is a setting where $ST$ has a unique pseudo-inverse, and it is not $\finv{T} \finv{S}_1$.
\end{enumerate}
\end{claim}
\begin{proof}
For part~1, let $w\in W$ be an element in the image of $ST$. Since $\inv{S}_1\in S|_{T(U)}\{1,2\}$, then by Lemma~\ref{lem: equivalent meaning of MP1-2 different spaces}, $v=\inv{S}_1(w)$ is a source for $w$ under $S$, and it is also in $T(U)$. Since $\inv{T}\in T\{1,2\}$, $\inv{T}(v)$ is a source of  $v$ under $T$. Therefore, $\inv{T} \inv{S}_1(w)$ is a source for $w$ under $ST$. If $w\in W$ is not in the image of $ST$, then $w'= S|_{T(U)} \inv{S}_1(w)$ is. This is true because, by Lemma~\ref{lem: equivalent meaning of MP1-2 different spaces}, $\inv{S}_1(w)$ is the source in $T(U)$ of some element $w'$ in the image of $ST$. Then
\begin{equation}
\inv{T} \inv{S}_1(w) = \inv{T} \inv{S}_1 S|_{T(U)} \inv{S}_1(w)= \inv{T} \inv{S}_1(w')\;.
\end{equation}
Namely, $\inv{T} \inv{S}_1(w)$ is the inverse value of some element in the image of $ST$, as required by Lemma~\ref{lem: equivalent meaning of MP1-2 different spaces}. Therefore, $\inv{T} \inv{S}_1\in (ST)\{1,2\}$, and if $T$ is onto, then $\inv{T} \inv{S}\in (ST)\{1,2\}$. 

For a counterexample when $T$ is not onto, let $U=\{u_1,u_2,u_3\}$, $V=\{v_1,v_2,v_3\}$, $W=\{w_1,w_2\}$, $T(u_1)=T(u_2) = v_2$, $T(u_3)=v_3$, $S(v_1)=S(v_3) = w_2$, $S(v_2)=w_1$. Defining $\inv{T}(v_3)=u_3$, $\inv{T}(v_1)=\inv{T}(v_2)=u_1$, $\inv{S}(w_1)=v_2$, and $\inv{S}(w_2)=v_1$ satisfies MP1--2, as may be easily verified by Lemma~\ref{lem: equivalent meaning of MP1-2 different spaces}. However, $\inv{T} \inv{S}(w_2) = u_1$, and the only source of $w_2$ under $ST$ is $u_3$, so $\inv{T} \inv{S}$ cannot be a $\{1,2\}$-inverse of $ST$.

For the second part, Let $U=\{u_1,u_2\}$, $\|u_2\| < \|u_1\|$, $V=\{v_1,v_2\}$, $\|v_1\|<\|v_2\|$, $W=\{w\}$, $S(v_1)=S(v_2)=w$, $T(u_1)=v_1$, and $T(u_2)=v_2$. Note that $T$ is a bijection. It is easy to verify that there is a unique valid choice for $\finv{T}$, namely, $\finv{T}(v_1)=u_1$ and $\finv{T}(v_2)=u_2$, and a unique valid choice for $\finv{S}_1$, namely, $\finv{S}_1(w)=v_1$. Similarly, there is a unique valid choice $\finv{(ST)}$ for a pseudo-inverse of $ST$, which is $\finv{(ST)}(w)=u_2$. Since $\finv{T} \finv{S}_1(w)=u_1$, $\finv{T} \finv{S}_1$ cannot be a pseudo-inverse of $ST$.
%\end{proofof}
\end{proof}

A specific composition of particular interest is that of a projection applied after an operator, even (perhaps especially) if the operator does not have a pseudo-inverse. One such case is that of a neural layer, which will be considered in Section \ref{sec:test cases}.
\begin{theorem}[Left-composition with projection]\label{thm:projection applied after operator}
Let $T:V\rightarrow W$, where $V$ and $W$ are subsets of normed spaces, let $\emptyset\neq C\subseteq T(V)$ be a Chebyshev set, and let $P_C:W\rightarrow W$ be the projection operator onto $C$. 
\begin{enumerate}
    \item 
    For $w\in W$, write $B_w=\arg\min\{\|v\|: v\in V,\,P_CT(v)=P_C(w)\}$, and let $W'\subseteq W$ be the set of elements $w\in W$ for which $B_w$ is well-defined. Then $w\in W'\setminus C$ implies $P_C(w)\in W'\cap C$, 
    and the following recursive definition yields a pseudo-inverse $\finv{(P_C T)}:W'\rightarrow V$:
    %%%
    \ignore{
    
    }
\begin{align*}
     \finv{(P_CT)}(w) = \begin{cases}
      \textup{some}\;v\in B_w,\hspace{13pt} w\in W'\cap C\\
      \finv{(P_C T)}(P_C(w)), w\in W'\setminus C\,.
    \end{cases}
\end{align*}

\item
If, in addition, $T$ is continuous, $W$ is a Hilbert space, and every closed ball is compact in $V$, then the pseudo-inverse of part~1 is defined on all of $W$ ($W'=W$).
\end{enumerate}
\end{theorem}
\begin{proof}
\begin{enumerate}
    \item 
    For any $w\in W$, BAS necessitates that
    \begin{align}
    \finv{(P_CT)}&(w)\in \arg\min\{\|v\|: v\in V\,\\
    &\textup{minimizes}\, \|P_CT(v)-w\|\}\;.
    \end{align}
    \ignore{
    
    }
Since $C\subseteq T(V)$, there is some $v_0$ s.t. $T(v_0)=P_C(w)$, and thus $P_CT(v_0)=P_C(w)$. Since $C$ is a Chebyshev set, $v\in V$ minimizes $\|P_CT(v)-w\|$ iff $P_CT(v)=P_C(w)$, and the BAS requirement becomes 
\begin{align}
\finv{(P_CT)}&(w)\in \arg\min\{\|v\|: v\in V,\,\\
&P_CT(v)=P_C(w)\}=B_w\;.   
\end{align}
\ignore{

}
If $w\in W'\cap C$, then our definition satisfies BAS and also MP2, since 
\begin{align}
\finv{(P_CT)}(P_CT)&\finv{(P_CT)}(w) \\&= \finv{(P_CT)}(P_C(w))\\
&=\finv{(P_CT)}(w)\;.
\end{align}
If $w\in W'\setminus C$, then by definition of $B_w$, $B_{P_C(w)}=B_w$. Thus, $B_{P_C(w)}$ is well-defined, so $P_C(w)\in W'\cap C$ and $\finv{(P_C T)}(P_C(w))$ is already defined. It satisfies BAS since it is in $B_{P_C(w)}$, and also satisfies MP2 since 
\begin{align}
    &\finv{(P_CT)}(P_CT)\finv{(P_CT)}(w) \\
    &= \finv{(P_CT)}(P_CT)\finv{(P_CT)}(P_C(w))\\ 
    &= \finv{(P_CT)}(P_C(w))\\
    &= \finv{P_C}(w)\;.
\end{align}
\ignore{

}
\item
A Chebyshev set in a Hilbert space is closed and convex, so by Theorem~\ref{thm:projection is continuous}, $P_C$ is continuous, so $P_CT$ is continuous as well. Therefore, for any $w\in C$, $A=(P_CT)^{-1}(\{w\})$ is closed. There is some $v_0\in V$ s.t. $P_CT(v_0)=w$, so $A_1=A\cap \{v\in V:\|v\|\leq \|v_0\|\}\neq \emptyset$. Since $A_1$ is compact, the norm function attains a minimum $m_w$ on that set by the extreme value theorem, and $B_w = \{v\in V: \|v\|=m_w\}$ is well-defined. For any $w\in W$, $B_w = B_{P_C(w)}$ is thus also well-defined, so $W'=W$.
\end{enumerate}
\end{proof}

\section{Test Cases}\label{sec:test cases}
In this section we consider nonlinear pseudo-inverses for specific cases. Recall that a valid pseudo-inverse $\finv{T}$ of an operator $T:V\rightarrow W$ should satisfy BAS and MP2. 
\subsection{Some One-dimensional Operators}
In Table~\ref{tab:pinv1d} we give the pseudo-inverses for some simple operators over the reals, with the Euclidean norm (the only norm up to scaling). The domain of the pseudo-inverse shows where it may be defined. 
We shall elaborate on the derivations for two of the operators and defer the explanations for the rest to the appendix. Below we write 
$\sgn_\epsilon(v) = \min\{1,\max\{-1,v/\epsilon\}\}$ for $\epsilon>0$, so $\sgn(v) = \lim_{\epsilon\rightarrow 0+}   \sgn_\epsilon(v)$, pointwise. To be clear, $\sgn(v)=\Ind{v>0}-\Ind{v<0}$.

\begin{table*}[htbp]
\caption{Pseudo-inverses of some one-dimensional operators, $T:\RR\rightarrow\RR$.}
\label{tab:pinv1d}
\centering
\scalebox{1}{
    \begin{tabular}{|c|c|c|c|}
    \hline
    $T(v)$ & $\finv{T}(w)$ & Domain of $\finv{T}$ & Unique\\
    \hline
    $v^2$  &  $\pm\sqrt{\max\{w,0\}}$ & $\RR$ & --\\
    $(v-a)^2$, $a\neq 0$  & $a-\sgn(a)\sqrt{\max\{w,0\}}$ & $\RR$ & +\\
    $\max\{v,0\}$ (ReLU)& $\max\{w,0\}$ & $\RR$ & +\\
    $v\cdot\Ind{|v|\geq a}$, $a\geq 0$ & $\sgn(w)\cdot\Ind{|w|>\frac{a}{2}}\max(a,|w|\}$ & $\RR$ & +\\
    $\sgn(v)\max\{|v|-a,0\}$, $a\geq 0$ & $\sgn(w)(|w|+a)$ & $\RR$ & +\\
    $\tanh(v)$  & $\arctanh(w)$ & $(-1,1)$& +\\
    $\sgn(v)$  & $0$  & $[-\frac{1}{2},\frac{1}{2}]$& +\\
    $\sgn_\epsilon(v)$, $\epsilon>0$ & $\epsilon \sgn_1(w)$ & $\RR$& +\\
    $\exp(v)$ & $\log(w)$  & $(0,\infty)$& +\\
    $\sin(v)$ &  $\arcsin(\min\{1,\max\{-1,w\}\})$ & $\RR$& +\\
    \hline
    \end{tabular}}
\end{table*}

Let $\bm{T(v)=(v-a)^2}$, $V=W=\RR$, $0\neq a\in\RR$. Note that $T(V)=[0,\infty)$, and for every $w\geq 0$, the sources of $w$ are $a\pm\sqrt{w}$, of which $a-\sgn(a)\sqrt{w}$ has the single smallest norm. For $w\notin T(V)$, the single closest element in $T(V)$ is $0$. Writing $W' = \{w\in W: \arg\min\{\|w_0-w\|:w_0\in T(V)\}\,\textup{is a singleton}\}$, we have $W'=\RR$. By Claim~\ref{cla:exist and unique}, a pseudo-inverse $\finv{T}:\RR\rightarrow \RR$ exists and it is unique. Directly from BAS, necessarily $\finv{T}(w) = a-\sgn(a)\sqrt{w}$, for $w\geq 0 $, and $\finv{T}(w)=a$, for $w<0$.

Let $\bm{T(v)=\max\{v,0\}}$ (ReLU), $V=W=\RR$. The ReLU function is a projection on the closed convex set $C=[0,\infty)$, which contains $0$, and by Corollary~\ref{cor:projection general inverse} it is its own unique pseudo-inverse, namely, $\finv{T}(w) = \max\{w,0\}$. Note that this generalizes easily to the operator $T:\RR^n\rightarrow \RR^n$, which applies ReLU entrywise, with the Euclidean norm. There we have a projection on the closed convex set $[0,\infty)^n$, and again by Corollary~\ref{cor:projection general inverse}, the operator is its own unique pseudo-inverse, namely, $\finv{T} = T$.

\subsection{A Neural Layer}
In this subsection we consider the pseudo-inverse of a single layer of a trained multilayer perceptron.

Let $V=\RR^n$ and $W=\RR^m$, let $A$ be an $m$ by $n$ matrix, and let $\sigma:\RR\rightarrow\RR$ be a transfer function. We will consider $\sigma$ to be either $\tanh$ or ReLU.
With some abuse of notation, we will also write $\sigma:W\rightarrow W$ for the elementwise application of $\sigma$ to a vector in $W$. Given $v\in V$, a neural layer operator $T:V\rightarrow W$ is then defined as $T(v) = \sigma(Av)$, which is continuous for both choices of $\sigma$ and smooth for $\tanh$. We consider the $L_2$ norm for both spaces and make the simplifying assumption that $A(V)=W$ (thus $m\leq n$). 

Note that Claim~\ref{cla:inv of composition} readily yields a $\{1,2\}$-inverse for this setting. In particular, for $\sigma=\textup{ReLU}$, we have that $\finv{A}\sigma\in (\sigma A)\{1,2\}$, since the pseudo-inverses of $A$ and $\sigma$ are also $\{1,2\}$-inverses, and entrywise ReLU is its own pseudo-inverse. We next consider pseudo-inverses.

\begin{claim}\label{cla:neural layer}
Let $T:V\rightarrow W$, where $V=\RR^n$, $W=\RR^m$, $m\leq n$, $T = \sigma A$, and $A$ is a full-rank matrix.
\begin{enumerate}
    \item
    If $\sigma = \tanh$, then $T$ has a unique pseudo-inverse, $\finv{T}:(-1,1)^m\rightarrow V$, defined as $\finv{T}(w) = \finv{A}(\arctanh(w))$.
   
 \item 
    Let $\sigma = \tanh$, let $1<k\in \NN$, and define $C_k = [-1+1/k,1-1/k]^m$ and $T_k = P_{C_k}T$. Then $\lim_{k\rightarrow\infty} T_k=T$, and we may define a valid pseudo-inverse $\finv{T}_k:W\rightarrow V$. 
\item
For $\sigma=\textup{ReLU}$, we can recursively define a valid pseudo-inverse $\finv{T}:W\rightarrow V$, where for $w\in W\setminus [0,\infty)^m$ we have $\finv{T}(w)=\finv{T}(\sigma(w))$, and for $w\in [0,\infty)^m$, $\finv{T}(w)$ is the solution 
to the quadratic program
$\min \|v\|^2$ s.t. $(Av)_i = w_i,$ $\forall i\; w_i>0,$ and 
$(Av)_i\leq 0,$ $\forall i\;w_i=0$.

\ignore{

}
\end{enumerate}
\end{claim}
\begin{proof}
\begin{enumerate}
    \item
    For $\sigma = \tanh$, $\sigma: W\rightarrow (-1,1)^m$ is a bijection, and $T(V)=(-1,1)^m$. Using Lemma~\ref{lem:inverse on the image}, $\finv{T}:T(V)\rightarrow V$ defined as $\finv{T}(w) = \finv{A}(\arctanh (w))$ is the unique valid pseudo-inverse.
    
    For $w\notin (-1,1)^m$ there is no nearest point in $T(V)$, so the pseudo-inverse cannot be defined.
    
    \item
    Note that for every $k>1$, $C_k$ is nonempty, closed, and convex, so it is a Chebyshev set. It is clear that $\lim_{k\rightarrow\infty} T_k=T$ in a pointwise sense, and thus $T_k$ may be seen as an approximation for $T$ for a large enough $k$. Note that $T$ is continuous, $C_k \subseteq T(V)$, $W$ is a Hilbert space, and every closed ball in $V$ is compact, so we may apply both parts of Theorem~\ref{thm:projection applied after operator} and define a valid pseudo-inverse $\finv{T}_k$ by
\begin{equation}
 \hspace{-0.7pt}\finv{T}_k(w) = \begin{cases}
      \textup{some}\;v\in B_w,\hspace{-6pt}& w\in C_k\\
     \finv{T}_k(P_{C_k}(w)), \hspace{-6pt}& w\in W\setminus C_k\;,
    \end{cases}
\end{equation}    
\ignore{
    
}

where $B_w=\arg\min\{\|v\|: v\in V,\,P_{C_k}T(v)=w\}$ is guaranteed to be well-defined for every $w\in C_k$. Thus, for $w\in W\setminus C_k$ we have $\finv{T}_k(w)=\finv{T}_k(P_{C_k}(w))$, and for $w\in C_k$, $\finv{T}_k(w)$ is a solution of the optimization problem
\begin{equation}
\hspace{-0.7pt}
\begin{aligned}
\textrm{min} \quad & \|v\|\quad \textrm{s.t.}&\hspace{-17pt}\\
 (Av)_i &= \arctanh(w_i)\;&\hspace{-17pt}\forall i\;| w_i|< 1-\frac{1}{k}\\
 (Av)_i&\leq \arctanh(\frac{1}{k}-1)\;&\hspace{-17pt}\forall i\;w_i= \frac{1}{k}-1 \\
 (Av)_i&\geq \arctanh(1-\frac{1}{k})\;&\hspace{-17pt}\forall i\;w_i= 1-\frac{1}{k}
\end{aligned}
\end{equation}
\ignore{

}
\ignore{

}
Note that it is equivalent to minimize $\|v\|^2$ instead of $\|v\|$. We further comment that this feasible constrained convex optimization problem has a unique solution.%
    
\item
For $\sigma=\textup{ReLU}$, $\sigma$ itself is a metric projection onto $[0,\infty)$, and when applied entrywise to $w\in W$, it projects $w$ onto $[0,\infty)^m$. Our assumption that $A(V)=W$ means that $T(V)=[0,\infty)^m$, namely, the image is nonempty, closed, and convex, and thus a Chebyshev set.
Again using Theorem~\ref{thm:projection applied after operator}, we can define a valid pseudo-inverse $\finv{T}:W\rightarrow V$, where for $w\in W\setminus [0,\infty)^m$ we have $\finv{T}(w)=\finv{T}(\sigma(w))$, and for $w\in [0,\infty)^m$, $\finv{T}(w)$ is a solution to the optimization problem
\begin{equation}
\begin{aligned}
\min \quad & \|v\|&\\
\textrm{s.t.} \quad & (Av)_i = w_i\;&\forall i\; w_i>0\\
& (Av)_i\leq 0\;&\forall i\;w_i=0
\end{aligned}
\end{equation}
which has a unique solution, as commented in part~2.%
\end{enumerate}
\end{proof}

\subsection{Wavelet Thresholding}
Let $V=W=\RR^n$, let $A$ be an $n$ by $n$ matrix representing the elements of an orthonormal wavelet basis, and let $\sigma:\RR\rightarrow\RR$ be some thresholding operator. Wavelet thresholding \citep{donoho1995adapting} takes as input a vector $v\in \RR^n$, which represents a possibly noisy image. The image is transformed into its wavelet representation $Av\in W$, and then the thresholding operator $\sigma$ is applied entrywise to $Av$. The result will be written as $\sigma(Av)$ or $(\sigma A)(v)$. Finally, an inverse transform is applied, and the result, $A^{-1}\sigma(Av)$, is the denoised image. For the thresholding operator, we will be particularly interested in hard thresholding, $\xi_a(x) = x\cdot\Ind{|x|\geq a}$, and soft thresholding, $\eta_a(x) = \sgn(x)\max\{|x|-a,0\}$, where $a\in\RR^+$ in both cases.

\begin{claim}\label{cla:wavelet hard th}
Let $T:V\rightarrow W$, where $T=\sigma A$.
\begin{enumerate}
    \item 
     If $\sigma$ has a pseudo-inverse $\finv{\sigma}:\RR\rightarrow\RR$, then $\finv{T}=A^{-1}\finv{\sigma}$ is a pseudo-inverse of $T$, where all the pseudo-inverses are w.r.t. the $L_2$ norm.
    \item
    If $\sigma = \xi_a$ for any $a\geq 0$, then $\finv{T}T = A^{-1}\sigma A$. Namely, wavelet hard thresholding is equivalent to applying $T$ and then its pseudo-inverse. In contrast, for $\sigma = \eta_a$ 
    (soft thresholding), $\finv{T}T = A^{-1}\sigma A$ does not hold for $a>0$.
\end{enumerate}
\end{claim}
\begin{proof}
\begin{enumerate}
    \item 
    Note first that the vector, or entrywise, view of $\sigma$, $\sigma:W\rightarrow W$, has a pseudo-inverse $\finv{\sigma}:W\rightarrow W$, which is exactly the entrywise application of $\finv{\sigma}$. The first part of the claim follows directly from applying part~2 of Claim~\ref{cla:comp invertible} to $\sigma A$. Note that $A^{-1}$ is norm-preserving under $L_2$ as an orthogonal matrix, and hence norm-monotone, and we take the identity operator as an isometry. 
    \item 
    From part~1, $\finv{T}T = A^{-1}\finv{\sigma}\sigma A$. The pseudo-inverse for both $\xi_a$ and $\eta_a$ is known (see Appendix~\ref{sec:1d pseudo inverse proofs}), so we can directly determine $\finv{\sigma}\sigma$. For $\sigma = \xi_a$, $\finv{\sigma}\sigma = \sigma$, so $\finv{T}T = A^{-1}\sigma A$. For $\sigma = \eta_a$ with $a>0$, however, $\finv{\sigma}\sigma(x) = x\cdot\Ind{|x|>a}\neq\sigma(x)$, so $\finv{T}T=A^{-1}\finv{\sigma}\sigma A\neq A^{-1}\sigma A$. The last inequality holds since there is some $u$ s.t. $\finv{\sigma}\sigma(u)\neq\sigma(u)$, so $\finv{\sigma}\sigma( A(A^{-1}u))\neq\sigma(A(A^{-1}u))$, and further applying the injective $A^{-1}$ to both sides keeps the inequality.
\end{enumerate}
\end{proof}

We note that denoising in general is ideally idempotent, and so is $\finv{T} T$ for any $T$, since $\finv{T} T \finv{T} T = \finv{T} T$ by MP2. 

\section{The Special Case of Endofunctions}
In the next three sections, we focus on the particular case where the domain and range of an operator are equal, namely, $T:V\rightarrow V$. Importantly, for such operators we may compose an operator with itself, and if $V$ is a vector space, even discuss general polynomials of an operator. We may of course apply our previous results to this particular setting, but will be able to introduce new and interesting types of inverses as well.

Let $V$ be a nonempty set. We denote by $OP(V)$ the set of all operators from $V$ to itself. The set $OP(V)$ is equipped with an associative binary operation (composition), making it a semigroup. Furthermore, it has a two-sided identity element, the identity operator $I$, making it a monoid. An element $T\in OP(V)$ has a two-sided inverse element $T^{-1}$ iff it is bijective. The composition operation is generally not commutative.   Integer powers of an operator are well-defined by induction: $T^n=T\circ T^{n-1}$, $T^1 = T$, and $T^0 = I$. If $T$ is bijective, then $T^{-n}$ is defined similarly using $T^{-1}$.

If $V$ is also a vector space over a field $F$,\footnote{We assume that the field is nontrivial, that is, $F\neq \{0\}$, and hence $0\neq 1$. }
then addition and scalar multiplication are defined in $OP(V)$ in a natural way: for every two operators $T_1$ and $T_2$, $(T_1+T_2)(v) = T_1(v)+T_2(v)$ for every $v\in V$, and for every $a\in F$, $(aT)(v) = aT(v)$. We denote by $0\in OP(V)$ the operator that maps all vectors to $0\in V$. With respect to these operations of addition and scalar multiplication, it is easy to see that $OP(V)$ is a vector space over $F$. Namely, the operation $+$ is commutative, associative, has an identity element ($0$), and has an inverse, $-T$ for every $T\in OP(V)$, and it holds that $a(T_1+T_2) = aT_1+aT_2$, $(a_1+a_2)T = a_1T+a_2T$, $(a_1a_2)T = a_1(a_2T)$, and $1T = T$.

Some additional properties are satisfied when $V$ is a vector space:
\begin{enumerate}
\item
  For every $a\in F$, $(aT_1)\circ T_2 = a(T_1\circ T_2)$.
\item
  Right distributivity: $(T_1+T_2)\circ T_3 = T_1\circ T_3+ T_2\circ T_3$. Importantly, left distributivity does \textit{not} generally hold: $T_1\circ (T_2+T_3) \neq T_1\circ T_2+ T_1\circ T_3$.
\item
  It holds that $0 \circ T = 0$, and if $T(0)=0$, then $T\circ 0 = 0$.
  \item
  For every polynomial $p\in F[x]$, $p(x) = \sum_{i=0}^m a_ix^i$, the operator $p(T)$ is well-defined, as satisfying $p(T)(v) = \sum_{i=0}^m a_i T^i(v)$ for every $v\in V$.
\end{enumerate}
The combined properties in the case where $V$ is a vector space make $OP(V)$ more than a right near-ring, not quite a near field, and seemingly not a standard algebraic structure. In particular, it is not an algebra or a ring, because it is not left distributive.

\subsection{A Motivating Example: Square Matrices}
Let $V$ be an $n$-dimensional vector space over a field $F$, and let $T$ be an $n\times n$ matrix over $F$ (equivalently, a linear transformation from $V$ to $V$). The \textit{characteristic polynomial} of $T$, $p_T(x)=\det(xI-T)$, is a degree $n$ monic polynomial whose constant term is $(-1)^n \det(T)$, and whose roots are the eigenvalues of $T$. The Cayley-Hamilton theorem states that $p_T$ \textit{vanishes in $T$}, namely, $p_T(T)=0$.
Writing $p_T(x)=\sum_{k=0}^n a_kx^k$, we thus have $\sum_{k=0}^n a_kT^k=0$, and rearranging, $T(\sum_{k=1}^n a_kT^{k-1})=(-1)^{n+1}\det(T) I$. If $T$ is invertible, we can apply $T^{-1}$ to both sides and divide by $(-1)^{n+1}\det(T)$, yielding
\begin{equation}
    T^{-1} = (-1)^{n+1} \det(T)^{-1}\sum_{k=1}^n a_kT^{k-1}\;.
\end{equation}
Namely, the inverse of an invertible matrix is a polynomial expression in that matrix. This form is convenient both theoretically and computationally, and one may ask whether such a result might be available for generalized inverses as well.

For the MP inverse the answer is in general negative. Specifically, for an $n\times n$ matrix $T$ over $F=\CC$, $T^\dagger$ is expressible as a polynomial in $T$ iff $T$ and $T^\dagger$ commute \citep{pearl1966generalized}.\footnote{Equivalently, the matrix $T$ is \textit{range-Hermitian}, that is, with equal images for $T$ and $T^*$.
}
Even for the more general case of a $\{1,2\}$-inverse, a polynomial expression does not necessarily exist. As shown by Englefield \citep{englefield1966commuting}, an $n\times n$ matrix $T$ over $\CC$ has a $\{1,2\}$-inverse expressible as a polynomial in $T$ iff $\rank{T}=\rank{T^2}$. 
One might even produce matrices $T$ for which a $\{1\}$-inverse cannot be polynomial in $T$, as follows.
\begin{claim}\label{cla:no poly pseudoinverse}\label{cla:1 inverse for nilpotent}
Let $T\neq 0$ be a square matrix over a field $F$ s.t. $T^k=0$ for some $k> 1$. If $X$ is a $\{1\}$-inverse of $T$, then it is not equal to any polynomial in $T$.
\end{claim}
\begin{proof}
Assume by way of contradiction that $X=p(T)$ for some polynomial $p \in F[x]$. By MP1, $(XT)^{k} = (XT)^{k-2}XTXT=(XT)^{k-1}$, and by induction, $(XT)^k=XT$. Since $X=p(T)$ and $T$ commute, $XT=(XT)^k=X^kT^k=0$. 
Again by MP1, this implies that $T=TXT=0$, which is impossible.  
\end{proof}
 One such matrix is the $n\times n$ left-shift matrix defined as $LS^{(n)}_{i,j}=\Ind{i=j+1}$ for $1\leq i,j\leq n$, where we require $n>1$. 
 
 This apparent shortcoming of $\{1\}$-inverses may be remedied by a different type of generalized inverse, the Drazin inverse, which is described next. Introduced by Drazin \citep{drazin1958pseudo} in the very general context of semigroups, it will be handy in the context of the semigroup $OP(V)$.

\subsection{The Drazin Inverse for Semigroups}
Let $\mathcal{S}$ be a semigroup, and let $x\in \mathcal{S}$. A \textit{Drazin inverse} of the element $x$, denoted $\drinv{x}$, is a $\{1^k,2,5\}$-inverse in the sense of satisfying the requirements:
\begin{align*}
&\textup{MP1$^k$: } x^k\drinv{x} x = x^k \\
&\textup{MP2: } \drinv{x} x \drinv{x} = \drinv{x} \\
&\textup{D5: } x\drinv{x} = \drinv{x} x\;,
\end{align*}
where MP1$^k$ is a variant of MP1 for some integer $k\geq 1$.\footnote{the original requirements in Drazin's work are slightly different, but equivalent.} 

The theorem below summarizes most of the main properties proved by Drazin for this generalized inverse. 
\begin{theorem}[\citep{drazin1958pseudo}]\label{thm:drazin}
Let $\mathcal{S}$ be a semigroup, and let $x\in\mathcal{S}$.
\begin{enumerate}
\item 
If $\drinv{x}$ exists, then it is unique and commutes with any element $y\in\mathcal{S}$ that commutes with $x$.
\item
If $\mathcal{S}$ has a unit and $x$ has an inverse $x^{-1}$, then $\drinv{x}=x^{-1}$.
\item 
If $\drinv{x}$ exists, then there is a unique positive integer, the index of $x$, $\drind{x}$, s.t. $x^{m+1}\drinv{x}=x^m$ for every $m\geq \drind{x}$, but for no $m<\drind{x}$.
\item 
If $\drinv{x}$ exists, and $k$ is a positive integer, then $\drinv{(x^k)}=(\drinv{x})^k$, and $\drind{x^k}$ is the unique positive integer $q$ that satisfies $0\leq kq - \drind{x}<k$.
\item 
If $\drinv{x}$ exists, then $\drind{\drinv{x}}=1$ and $\drinvdrinv{x} = x^2\drinv{x}$.
\item 
It holds that $\drinvdrinv{x}=x$ iff $x$ is Drazin-invertible with index $1$.
\item 
If $x$ is Drazin-invertible, then $\drinvdrinv{(x^k)}=x^k$ for every integer $k\geq\drind{x}$.
\item 
The element $x$ is Drazin-invertible iff there exist $a,b\in\mathcal{S}$ and positive integers $p$ and $q$ s.t. $x^p=x^{p+1}a$ and $x^q=bx^{q+1}$. If the condition holds, then $\drinv{x}=x^Ma^{M+1}=b^{M+1}x^M$, where $M=\max\{p,q\}$, and $\drind{x}\leq \min\{p,q\}$.
\item 
If $\mathcal{S}$ has a (two-sided) zero, and $x^k=0$ for some positive integer $k$, then $\drinv{x}=0$.
\end{enumerate}
\end{theorem}
Since $\drinv{x}$ is unique when it exists, one may refer to \textit{the} Drazin inverse of $x$.

\subsubsection{Application to Matrices}
As shown already by Drazin (\citep{drazin1958pseudo}, Corollary~5), when the semigroup $\mathcal{S}$ is also a finite-dimensional algebra, $\drinv{x}$ always exists and is a polynomial in $x$. This applies in particular to square matrices over a field.

For a matrix $T$, one can give a polynomial expression for the Drazin inverse using the \textit{minimal polynomial} of $T$, $m_T(x)$, which is the unique monic polynomial with the least degree that vanishes in $T$. The minimal polynomial exists for every matrix $T$ and divides every other polynomial that vanishes in $T$. 

To see this, we can write $m_T(x) = cx^l(1-xq(x))$, where $c\neq 0$, $l\geq 0$, and $q$ is a polynomial. Thus, $0=m_T(T)=cT^l-cT^{l+1}q(T)$, and therefore, $T^l=T^{l+1}q(T)$. If $l=0$, then we may multiply both sides by $T$, so in any case, $T^k=T^{k+1}q(T)=q(T)T^{k+1}$, where $k=\max\{l,1\}\geq 1$. By part~8 of Theorem~\ref{thm:drazin}, 
\begin{equation}
    \drinv{T} = T^kq(T)^{k+1}\;.
\end{equation}
For more on the Drazin inverse for matrices, see, e.g., \citep{ben2003generalized}.

 In summary, for any field, there are matrices for which a $\{1\}$-inverse cannot be a polynomial in the matrix, while the Drazin inverse always is, making their values necessarily different. This is true in the particular case of matrices over the complex numbers and the MP inverse. One such example is the $n\times n$ right-shift matrix, defined by $RS^{(n)}_{ij}=\Ind{i+1=j}$ for $1\leq i,j\leq n$, whose MP inverse (and conjugate transpose) is the already mentioned left-shift matrix $LS^{(n)}$. By Claim~\ref{cla:1 inverse for nilpotent} (or simple observation), $LS^{(n)}$ cannot be a polynomial in $RS^{(n)}$ for $n>1$. The Drazin inverse of $RS^{(n)}$ equals $0$, as it is for any nilpotent element (Theorem~\ref{thm:drazin} part~9). 

 \section{Vanishing Polynomials of Nonlinear Operators}
\ignore{
}
\ignore{
}

\ignore{ 
}

 We wish to prove results regarding generalized inverses for nonlinear endofunctions that are in a similar spirit to those shown for matrices. In other words, to express generalized inverses by forward applications of the operator at hand, specifically, using polynomials. 
 
  Inversion is generally hard to compute, certainly for nonlinear operators. Applying a forward operator, on the other hand, is straightforward. This task is directly related to inverse problems and to learning. For example, a neural network for some image manipulation task is a highly complex nonlinear operator $T:V\rightarrow V$. An intriguing question is: \emph{Can we approximate a generalized inversion of the network by a polynomial containing only forward (inference) applications of the network?} This motivates our study here.
 
 A topic that is relevant to our investigation is that of vanishing polynomials of nonlinear operators, which are interesting in their own right, and will be discussed in depth in this section. Generalized inverses that build on the results shown in this section will be defined in the next section.

In this section, unless otherwise said, we assume that $V$ is a vector space over $F$. As already mentioned, for every $p\in F[x]$ and $T\in OP(V)$, the operator $p(T)$ is well defined. The transformation $p\rightarrow p(T)$ is linear, namely, if $p,q \in F[x]$, where $p(x)=\sum_{i=0}^n a_i x^i$, $q(x)=\sum_{i=0}^m b_i x^i$, and $a\in F$, then $(p+q)(T) = p(T) + q(T)$ and $(ap)(T) = ap(T)$. We also have, even without left distributivity, that
    \begin{align}
        (pq)(T) &= \left(\sum_{i=0}^m\sum_{j=0}^na_jb_ix^{i+j}\right)(T)\\
        &= \sum_{i=0}^m\sum_{j=0}^na_jb_iT^{i+j} \\
        &=\sum_{i=0}^mb_ip(T)\circ T^i\;.
    \end{align}

The set of \textit{homogeneous operators}, namely, operators with $0$ as a fixed point, form a vector subspace of $OP(V)$ and are closed under composition. Consequently, for any $p\in F[x]$ and homogeneous $T$, $p(T)$ is homogeneous.

\begin{definition}[Vanishing polynomial]
A polynomial $p\in F[x]$ is a vanishing polynomial of $T\in OP(V)$ if $p\neq 0$ and $p(T)=0$.\footnote{Not to be confused with a vanishing polynomial of a ring, which evaluates to zero on every element of the ring.}
\end{definition}
We will sometimes say that a polynomial vanishes in or for an operator in the same sense.
We may already highlight some elementary properties of vanishing polynomials. 
\begin{claim}\label{cla:elem prop vanishing}
Let $T\in OP(V)$, $p\in F[x]$ vanishing in $T$, $p(x) = \sum_{i=0}^m a_ix^i$.
\begin{enumerate}
\item 
If $a_0\neq 0$, then $T$ is injective (equivalently, left invertible), and a left inverse is given by $S=-a_0^{-1}\sum_{i=1}^m a_iT^{i-1}$. If $T$ is also surjective, then it is a bijection, and $S=T^{-1}$.
\item 
 For any $q\in F[x]$, $(pq)(T)(v)=0$ for every $v\in V$, and if $q\neq 0$, then $pq$ is vanishing in $T$.
 \item 
 If $p_i$ is vanishing in $T_i$, $i=1,\ldots,k$, then $\Pi_{i=1}^k p_i$ is a common vanishing polynomial of $\{T_i\}_i$.
 \item 
For every $p_1\in F[x]$, $p_1(T) = r(T)$ for $r\in F[x]$, where either $r=0$ or $deg(r)<deg(p)$.
\item
For any invertible linear $P:V\rightarrow V$, $p$ is vanishing in $P^{-1}TP$.
\end{enumerate}
\end{claim}
\begin{proof} 
\begin{enumerate}
\item 
By right distributivity, 
\begin{align}
ST&=(-a_0^{-1}\sum_{i=1}^m a_iT^{i-1})\circ T \\
&= -a_0^{-1}(p(T) - a_0 I) = I\;.
\end{align}
    If $T$ is surjective, then $T^{-1}$ exists, and $ST=I$ implies $S=T^{-1}$.
\item
If $q(x)=\sum_{i=0}^l b_ix^i$, then 
\begin{align}
(pq)(T)(v) &=\left(\sum_{i=0}^l b_ip(T)\circ T^i\right)(v)\\
&=\sum_{i=0}^l b_ip(T)(T^i(v))=0\;.
\end{align}
If $q\neq 0$, then $pq\neq 0$, satisfying the conditions for a vanishing polynomial.
\item 
Immediate from the previous part.
\item 
If $p_1=0$ the claim holds, so assume $p_1\neq 0$. Divide $p_1$ by $p$ with remainder, yielding $p_1=pq+r$, where $q,r\in F[x]$ and either $r=0$ or $deg(r)<deg(p)$. Since $(pq)(T)=0$, $p_1(T)=r(T)$.
\item 
For every $v\in V$,
\begin{align}
p(P^{-1}&TP)(v) = \sum_{i=0}^m a_i(P^{-1}TP)^i(v)\\
&=\sum_{i=0}^m a_iP^{-1}T^iP(v)\\
&=P^{-1}\left(\sum_{i=0}^m a_iT^iP(v)\right)\\
&=P^{-1}(0)=0\;.
\end{align}
Note that $T=P(P^{-1}TP)P^{-1}$, so the argument works in the other direction as well.
\end{enumerate}
\end{proof}

\subsection{The Minimal Polynomial}
For a square matrix $T$, there always exists a unique monic vanishing polynomial with minimal degree that divides any vanishing polynomial of $T$, such as the characteristic polynomial, without remainder. We show that such an object also exists for nonlinear operators with a vanishing polynomial.
\begin{theorem}[Minimal polynomial] If $T\in OP(V)$ has a vanishing polynomial, then there exists a unique monic polynomial $p_m\in F[x]$ that is vanishing in $T$ and has minimal degree. Furthermore, for every vanishing polynomial $p$ of $T$, $p_m | p$.
\end{theorem}
\begin{proof}
Let $k$ be the minimal degree for a vanishing polynomial of $T$, and let $p_m$ be such a polynomial, w.l.o.g. monic, $p_m(x) = \sum_{i=0}^k a_ix^i$ with $a_k=1$. Let $p$ be vanishing in $T$, so $p = p_mq + r$, where $q,r\in F[x]$, $q(x) = \sum_{i=0}^l b_ix^i$, and either $r=0$ or $deg(r)<k$. Thus, $p - r =  p_mq$.
We have on the one hand that for every $v\in V$, $(p-r)(T)(v) = p(T)(v) - r(T)(v) = -r(T)(v)$. On the other hand, by Claim~\ref{cla:elem prop vanishing} part~2, $(p_mq)(T)(v)=0$. Therefore, $r(T)(v) = 0$ for every $v$. If $r\neq 0$, then it would be a vanishing polynomial of $T$ with $deg(r)<deg(p_m)$, which is impossible. Therefore $r=0$, and $p_m | p$.

If $p'_m$ is another monic vanishing polynomial of $T$ with minimal degree, then $p'_m = h p_m$, where $h\in F[x]$ and $deg(h)=0$, and since both $p_m$ and $p'_m$ are monic, $h=1$.
\end{proof}

It has already been shown that a nonzero free coefficient of a vanishing polynomial of $T$ implies that $T$ is injective. In the setting of square matrices, the free coefficient of the minimal or characteristic polynomial of $T$ (where its absolute value is $\det(T)$) is nonzero iff the operator is bijective. In that setting, bijective, injective, or surjective are equivalent. The following claim recovers some of these connections using the minimal polynomial. 

\begin{claim}\label{cla:surjective minimial poly}
Let $T\in OP(V)$ have a vanishing polynomial, and let $p_m(x)=\sum_{i=0}^m a_i x^i$ be its minimal polynomial.
If $T$ is surjective, then $a_0\neq 0$ and $T$ is bijective/invertible.
\end{claim}
\begin{proof}
Assume by way of contradiction that $a_0 = 0$. Then for every $v\in V$, $\left(\sum_{i=1}^m a_i T^{i-1}\right)\circ T(v)=0$. Since $T$ is surjective, then for every $u\in V$ it holds that $\left(\sum_{i=1}^m a_i T^{i-1}\right)(u)=0$. We have found a vanishing polynomial of $T$ with degree smaller than that of $p_m$, which is impossible. Thus, $a_0\neq 0$, and by part~1 of Claim~\ref{cla:elem prop vanishing}, $T$ is also injective, and thus bijective. 
\end{proof}

\subsection{Existence of Vanishing Polynomials}
In this subsection we show cases where operators that may be nonlinear have vanishing polynomials. In addition, we derive vanishing polynomials in compound scenarios, such as rational powers of operators and products of operators, where the constituents have vanishing polynomials.

\subsubsection{Examples by Possible Degree}\label{subsubsec:examples by degree}
There are no operators with a vanishing polynomial of degree $0$ (for $V\neq \{0\}$), and those with a degree $1$ polynomial must be of the form $a I$. We may give examples of vanishing polynomials for nonlinear operators with any degree $m>1$. 

For $m=2$, we have all idempotent operators, with the polynomial $x^2-x$, including nonlinear projections and operators of the form $\Ind{v\in A}\cdot v$ for some $A\subset V$. 

For any $m\geq 2$, we can consider $m-1$ strictly nested nonempty closed and convex sets in a Hilbert space to obtain a vanishing polynomial $x^{m}-x^{m-1}$. The operator simply projects a point to the next convex set.
\ignore{}

More generally, for any integer $0\leq k<m$, we can obtain a vanishing polynomial $x^m-x^k$ for a scenario where $V = \mathop{\dot{\bigcup}}_{i=1}^m A_i$, $T|_{A_i}$ is a bijection from $A_i$ to $A_{i+1}$ for $1\leq i < m$, and $T|_{A_m}=T|_{A_{k+1}}^{-1}\circ \ldots \circ T|_{A_{m-1}}^{-1}$, taking $T|_{A_m}=I$ for $k=m-1$. (Note that even for a plain set $V$, this construction still yields $T^m=T^k$.) If we defined $T|_{A_m} = 0$ and $0\in A_m$, we would have a vanishing polynomial $x^m$.

\subsubsection{Contraction Mappings}
Let $V$ be a metric space as well a vector space, let $\emptyset\neq A\subset V$ be bounded, and let $T\in OP(V)$ be a contraction mapping on $A$. Clearly, the diameter of $T^k(A)$ becomes as small as we wish as $k$ grows. If for some $0\neq p\in F[x]$, $p(T)(u)=0$ for every $u\in T^k(A)$, then  $(p(T)T^k)(u)=0$ for every $u\in A$. If $T^l(V)\subseteq A$ for some $l$, then $p(x)x^{k+l}$ is a vanishing polynomial of $T$.

For example, let $V$ be $\RR^n$ with the Euclidean norm, let $r_1>r_2>0$, and let $0\leq a<1$. Any linear mapping $S$ from $\RR^n$ to itself, whose singular values are all smaller than $a$ in absolute value, is a contraction mapping with parameter $a$. 
Let $S_v$ be such a contraction mapping chosen arbitrarily for each $v\in V$. Define $T(v)$ as follows. If $v\notin \bar{B}(0,r_1)$, pick an arbitrary value in $\bar{B}(0,r_1)$; for $v\in \bar{B}(0,r_1)\setminus \bar{B}(0,r_2)$, $T(v)=S_v(v)$; otherwise, $T(v)=S_0(v)$. By the above argument, $T$ has a vanishing polynomial.

\subsubsection{Finite Images and Dimensions}
When the image of an operator has finite cardinality (possibly after several applications), then it has a vanishing polynomial. This applies to any actual implementation of operators, since all practical data types are discrete. 

\begin{claim}\label{cla: gen for finite num vals}
Let $T:V\rightarrow V$ be an operator, and let $l\geq 0$ be an integer s.t. $|T^l(V)|=|T^{l+1}(V)|=m<\infty$. If $V$ is a vector space, then $T$ has a vanishing polynomial of degree at most $m^2+l$. If $V$ is a plain set, then $T^{m!+l}=T^l$.
\end{claim}
\begin{proof}
Let $T^{l}(V) = \{v_1,\ldots,v_m\}$, so for every $i\geq 1$, $T^{i-1}$ permutes this set. Assume $V$ is a vector space. Define an $m^2+1$ by $m^2$ matrix $A$ by
\begin{align}
    A_{i,(j-1)m+k} &= \begin{cases}
      1, & T^{i-1}(v_j) = T(v_k) \\
      0, & \text{otherwise,}
    \end{cases}
\end{align}
where $1\leq j,k\leq m$. We thus have for every $1\leq i\leq m^2+1$ and $1\leq j\leq m$ that 
\begin{equation}
    T^{i-1}(v_j) = \sum_{k=1}^m A_{i,(j-1)m+k}T(v_k) \;.
\end{equation}
Since $A$ has more rows than columns, there is a vector $0\neq a\in F^{m^2+1}$ s.t. $a^\top A=0$, and for every $1\leq j \leq m$ we have that 
\begin{align}
\sum_{i=1}^{m^2+1}& a_i T^{i-1}(v_j) \\&= \sum_{i=1}^{m^2+1} a_i \sum_{k=1}^m A_{i,(j-1)m+k}T(v_k) \\
&= \sum_{k=1}^m \left(\sum_{i=1}^{m^2+1} a_i A_{i,(j-1)m+k}\right)T(v_k)\\
&=0\;.
\end{align}
For every $v\in V$ there is some $1\leq j \leq m$ s.t. $T^l(v)=v_j$, and hence $T^{l+i-1}(v)=T^{i-1}(v_j)$ for every $i\geq 1$. Therefore, $\sum_{i=1}^{m^2+1} a_i T^{l+i-1}(v)=\sum_{i=1}^{m^2+1} a_i T^{i-1}(v_j)=0$,
proving the first part.

If $V$ is a mere set, then we still have that $T^{i}$ permutes $\{v_1,\ldots,v_m\}$ for $i\geq 0$, corresponding to some permutation in the symmetric group $S_m$. By Lagrange's theorem, $T^{m!}(v_j)=I(v_j)$, for every $1\leq j\leq m$. For every $v\in V$, $T^l(v) = v_j$ for some $j$, so $T^{m!+l}(v)=T^l(v)$.
\end{proof}

\begin{corollary}\label{cor:finite values}
Every operator with a finite number of values has a vanishing polynomial. 
\end{corollary}

\begin{corollary}
If $V$ is finite, then every operator has a vanishing polynomial. 
\end{corollary}

\begin{claim}
Let $T\in OP(V)$. The span of $\{T^i:i\geq 0\}$ has a finite dimension over $F$ iff $T$ has a vanishing polynomial.
\end{claim}
\begin{proof}
If the dimension is $k\in \NN$, then $I,T,\ldots,T^{k}$ are linearly dependent, yielding a vanishing polynomial.

If $T$ has a vanishing polynomial $p$, then for every $p_1\in F[x]$, $p_1(T) = r(T)$ for $r\in F[x]$, where either $r=0$ or $deg(r)<deg(p)$ (Claim~\ref{cla:elem prop vanishing} part~4).  Therefore, the span of $\{T^i:i\geq 0\}$ has a finite dimension.
\ignore{}
\end{proof}

Finally, one may show a nonexistence result for all polynomial functions over an infinite field, such as $\RR$. Members of this important family of approximation functions over $\RR$ thus never have a vanishing polynomial. At the same time,  discrete functions with a finite number of values, another important family of approximation functions, have been shown to always have a vanishing polynomial (Corollary~\ref{cor:finite values}). It is therefore possible to approximate operators with others that have vanishing polynomials, even if many do not have a vanishing polynomial themselves.

\begin{claim}
Let $p:F^1\rightarrow F^1$ be a polynomial with $deg(p)>1$, and let $|F|=\infty$. Then $p$ does not have a vanishing polynomial. 
\end{claim}
\begin{proof}
Assume, by way of contradiction, that $q\in F[x]$ is vanishing in $p$. Since $p,q\in F[x]$, it holds that $q(p)\in F[x]$.
Write $q(x)=\sum_{i=0}^n b_ix^i$ and $p(x)=\sum_{i=0}^m a_ix^i$, where $m=deg(p)$ and $n=deg(q)\geq 1$. 
The coefficient of the highest power in $q(p)$ is $b_na_m^n\neq 0$, so $q(p)\neq 0$. Note that when $q(p)$ is evaluated on $v\in F^1$, its free coefficient becomes linear in $v$. However, since $deg(p)>1$, the highest power of $v$ in $q(p)(v)$ cannot be cancelled this way as can happen for a linear $p$ (which indeed has a vanishing polynomial).

Thus, $q(p)(v)$ is a nonzero polynomial in $v$, and the number of its roots is bounded by its degree. Since $|F|=\infty$, there must be $v\in F^1$ with $q(p)(v)\neq 0$, contradicting the fact that $q$ is vanishing in $p$.
\end{proof}

\subsubsection{Affine Mappings}
If $T(v)=A(v)+b$, where $A$ is linear and $b\in V$, then for any polynomial $p(x)=\sum_{i=0}^m a_0x^i$ and $v\in V$, $p(T)(v)=p(A)(v)+c$, where $c\in V$ is independent of $v$. Therefore,
\begin{align}
p^2(T)(v) &= \left(\sum_{i=0}^m a_ip(T)\circ T^i\right)(v)\\
&=\sum_{i=0}^m a_ip(T)(T^i(v))\\
&=\sum_{i=0}^m a_i(p(A)(T^i(v))+c) \\
&=\sum_{i=0}^m a_i c + p(A)\left(\sum_{i=0}^m a_iT^i(v)\right) \\
&= p(1)c + p(A)(p(T)(v))\;.
\end{align}
If, in addition, $p$ vanishes in $A$, then 
\begin{align}
    p^2(T)(v)=p(1)c = p(1)p(T)(v)\;,
\end{align}
yielding the following:
\begin{claim}
Let $T(v)=A(v)+b$, where $A$ is linear and $b\in V$. If $p$ is vanishing in $A$ and $deg(p)\geq 1$, then $p^2-p(1)p$ is vanishing in $T$.
\end{claim}
\begin{corollary}
If $A\in F^{n\times n}$ is a matrix and $b\in F^n$, then $T\in OP(F^n)$, defined as $T(v) = Av+b$, has a vanishing polynomial.
\end{corollary}
\subsubsection{Rational Powers}
We start by considering the inverse of a bijection.

\begin{definition}[Reciprocal polynomial]
If $p\in F[x]$, $p(x)=\sum_{i=0}^m a_i x^i$, then the reciprocal polynomial of $p$ is defined as $p^*(x) = \sum_{i=0}^m a_i x^{m-i}$.
\end{definition}
 Informally, $p^*(x)$ equals $x^m p(1/x)$. For an invertible matrix $A$ with characteristic polynomial $p_A$, it is known that $p_{A^{-1}} = p_A(0)^{-1} p_A^*$.
\begin{claim}\label{cla:reciprocal vanishing polys}
If $T$ is invertible and $p(x)=\sum_{i=0}^m a_i x^i$ is vanishing in $T$, then $p^*(x)$ is vanishing in $T^{-1}$. If $p_m$ is the minimal polynomial of $T$, then $p_m(0)^{-1}p_m^*$ is the minimal polynomial of $T^{-1}$.
\end{claim}
\begin{proof}
For every $v\in V$, $p(T)(T^{-m}(v))=0$, so
$\sum_{i=0}^m a_iT^{i-m}(v)=0$. Equivalently, $p^*(T^{-1})(v)=0$, as required. By Claim~\ref{cla:surjective minimial poly}, since $T$ is surjective, we have that $p_m(0)\neq 0$, and thus, 
$p_m(0)^{-1}p_m^*$ is a monic polynomial and vanishing in $T^{-1}$. 
By the same argument, if $q_m$ is the minimal polynomial of $T^{-1}$, then $q_m(0)^{-1}q_m^*$ is monic and vanishing in $T$. Therefore, $deg(q_m)=deg(p_m)$, and $p_m(0)^{-1}p_m^*$ is in fact the minimal polynomial of $T^{-1}$.
 \end{proof}

 \begin{claim}\label{cla: vanishing for rational powers}
 Let $T$ have a vanishing polynomial $p(x)=\sum_{i=0}^m a_i x^i$, and let $T_1$ be an operator satisfying $T_1^l = T^k$, where $k,l\in\NN$, $l>0$ (informally, $T_1 = T^{k/l}$). Then $T_1$ has a vanishing polynomial of degree at most $m l$.
 \end{claim}
 \begin{proof}
If $p_1(x)=\sum_{i=0}^n b_ix^i$ vanishes in $T^k$, then 
$\sum_{i=0}^n b_ix^{il}$ vanishes in $T_1$. Thus, 
it suffices to find a vanishing polynomial for $T^k$ of degree at most $m$.

  For $k=0$ or $m=0$ the claim is obvious, so assume $k,m\geq 1$. For every $0\leq j \leq m$, we have $x^{jk} = q_j(x)p(x)+r_j(x)$, where $q_j,r_j\in F[x]$ and either $r_j=0$ or $deg(r_j)<m$. The set $\{r\in F[x]: deg(r)<m\}\cup\{0\}$ is a vector space that is isomorphic to $F^m$, so $r_0, \ldots, r_{m}$ are linearly dependent, namely, $\sum_{j=0}^{m} \alpha_j r_j=0$ for some $\alpha_0,\ldots,\alpha_m\in F$ that are not all $0$. Therefore,  $\sum_{j=0}^{m} \alpha_j (T^{k})^{j}=\sum_{j=0}^{m} \alpha_j ((q_jp)(T)+r_j(T))=0$, and we are done.

 \end{proof}
 \begin{corollary}
     If $T$ is bijective, with a degree $m$ vanishing polynomial, and $T_1 = T^{k/l}$ where $k\in\ZZ$, $l\in\NN$, $l>0$, then $T_1$ has a vanishing polynomial of degree at most $m l$.
 \end{corollary}

\subsubsection{Product Operators}
Given several vector spaces over the same field and operators over each, we may define the product operator on the product, or direct sum, of the spaces. It is easy to observe that if each operator has a vanishing polynomial, then so does the product operator. 
\begin{claim}
Let $T_i:V_i\rightarrow V_i$, $i=1,\ldots,n$, where $V_i$ are vector spaces over $F$, and define $T:\bigoplus_i V_i\rightarrow \bigoplus_i V_i$ as the componentwise application of $T_i$. If $p_i$ is vanishing for $T_i$, then $\Pi_i p_i$ is vanishing for $T$. 
\end{claim}

\subsection{Interpretation as a Linear Operator on a Nonlinear Embedding}
Let $T$ have a vanishing polynomial $p$ of degree $n\geq 1$. We can always make $p$ monic, so w.l.o.g., $p(x)=x^n+\sum_{i=0}^{n-1} a_i x^i$. Define an embedding $\varphi_T:V\rightarrow \bigoplus_{i=1}^n V$, by $\varphi_T(v)=(v,T(v),\dots,T^{n-1}(v))^\top$. Thus, $\varphi_T(T(v))=(T(v),\ldots,T^{n-1}(v),-\sum_{i=0}^{n-1}a_iT^i(v))^\top = C_p^\top \varphi_T(v)$, where $C_p$ is the \textit{companion matrix} of $p$, 
\begin{align}
  C_p=
\begin{pmatrix}
  0 & 0 & \ldots & 0 & -a_0\\
  1 & 0 & \ldots & 0 & -a_1\\
  0 & 1 & \ldots & 0 & -a_2\\
  \vdots & \vdots & \ddots & \vdots & \vdots\\
  0 & 0 & \ldots & 1 & -a_{n-1}\\  
\end{pmatrix}
\end{align}
A known property of $C_p$ is that its characteristic and minimal polynomials are equal to $p$. 

Thus, the possibly nonlinear $T$ may be represented as a linear transformation from $\bigoplus_{i=1}^n V$ to itself, restricted to $\varphi_T(V)$. 

\subsection{Vanishing Polynomials and Eigenvalues}
For a linear operator $T$, the characteristic polynomial is a vanishing polynomial in $T$ and is also solved by the eigenvalues of $T$. In this subsection we examine the existence of such a phenomenon for nonlinear operators. 

For $T\in OP(V)$, $0\neq v\in V$ is an eigenvector of $T$ with eigenvalue $\lambda\in F$, if $T(v)=\lambda v$. Throughout this subsection we will assume $V\neq \{0\}$. For the special values $\lambda = 0,1$ we can draw a general connection to the roots of a vanishing polynomial. 
\begin{claim}\label{cla:0 1 eigenvalues}
Let $T\in OP(V)$ have a vanishing polynomial $p$. 
\begin{enumerate}
\item 
If $T$ has an eigenvector with $\lambda=1$ (a nonzero fixed point), then $p(\lambda)=0$.
\item    
If $T$ is homogeneous and has an eigenvector with $\lambda=0$, then $p(\lambda)=0$. 
\end{enumerate} 
\end{claim}
\begin{proof}
\begin{enumerate} 
\item
If $v$ is a fixed point of $T$, then $T^k(v) = v$ for every $k\geq 0$, so $0 = p(T)(v) =  (\sum_{i=0}^m a_i)v = p(1)v$. Since $v\neq 0$, $p(1)=0$.
\item 
If $v$ is an eigenvector with $\lambda=0$, then $T^k(v) = 0$ for every $k\geq 1$, so $0 = p(T)(v) = a_0 v$. Since $v\neq 0$, it follows that $a_0=0$, and hence $p(0)=0$.

\end{enumerate} 
\end{proof}
This result of course holds for the minimal polynomial in particular.

To illustrate, let $V$ be a Hilbert space, and let $P_C$ be the projection onto a nonempty closed convex set $C$. The polynomial $p(x)=x^2-x$ is vanishing in $P_C$ and also minimal, every $0\neq v\in C$ is an eigenvector of $P_C$ with $\lambda=1$, and indeed $p(1)=0$. Even though $0$ is also a root of $p$, it is not necessarily an eigenvalue, for example, if  $0\notin C$. To demonstrate the second part of the claim, for any vector space $V$ and a set $A$, $\{0\}\subsetneq A\subsetneq V$, define $T$ as $T(v)=\Ind{v\in A}\cdot v$. This operator too has $p(x)=x^2-x$ as a vanishing and minimal polynomial. Every $0\neq v\notin A$ is an eigenvector with $\lambda=0$ and we have that $p(0)=0$. In addition, every $0\neq v\in A$ is an eigenvector with eigenvalue $1$, so here all roots of $p$ are eigenvalues.

In general, though, eigenvalues are not necessarily roots of the minimal polynomial, even in elementary cases, such as nilpotent operators. For nilpotent operators, $p(x)=x^m$ is a minimal polynomial for some $m\geq 1$, and any $0\neq v\in T^{m-1}(V)$ is an eigenvector with $\lambda=0$, which is indeed a root of $p$. However, even for $V=\RR$,  with $T(v)=\Ind{v\in (0,a)}(v+1)$ for some $0<a<1$, there are also eigenvectors for every $\lambda\in (1/a+1,\infty)$.

\ignore{
}

One case where we may hope for a more intricate relation between eigenvalues and roots of a vanishing polynomial of $T$ is that of $\gamma$-homogeneous operators.

\begin{definition}[$\gamma$-homogeneity]
\label{def:k-hom}
Let $T\in OP(V)$. Given $\gamma\in\NN$, $T$ is $\gamma$-homogeneous if for every $0\neq a\in F$ and $v\in V$, $T(av) = a^\gamma T(v)$. If $F=\RR\text{ or }\CC$ and $\gamma\in[0,\infty)$, $T$ is absolutely $\gamma$-homogeneous if $T(av) = |a|^\gamma T(v)$ for every $a\neq 0$ and $v\in V$. If $F=\RR$ and $\gamma\in[0,\infty)$, $T$ is positively $\gamma$-homogeneous if $T(av) = a^\gamma T(v)$ for every $a>0$ and $v\in V$.
\end{definition}
Note that if $a^\gamma\neq 1$ for some $0\neq a\in F$, then a $\gamma$-homogeneous $T$ satisfies $T(0)=0$. If $\gamma\neq 0$, then $T(0)=0$ is implied by absolute and positive $\gamma$-homogeneity. We may state the following.
\begin{claim}\label{cla:homgeneous eigenvalues}
Let $T\in OP(V)$ have a vanishing polynomial, and a minimal polynomial $p$.
\begin{enumerate}
    \item
    If $T$ is nilpotent and $\gamma$- or absolutely $\gamma$-homogeneous, then its only eigenvalue is the only root of $p$, namely, $0$.
    \item 
     If $T$ is \emph{not} nilpotent and $F$ is infinite, then $T$ cannot be $\gamma$-, absolutely $\gamma$-, or positively $\gamma$-homogeneous if $\gamma\notin \{0,1\}$.
    \item 
    Let $\gamma=1$. If $T$ is $\gamma$-homogeneous and $|F|>2$, then any eigenvalue $\lambda$ of $T$ satisfies $p(\lambda)=0$.  If $T$ is absolutely $\gamma$-homogeneous, then $p(|\lambda|)=0$ for any eigenvalue $\lambda$, and in addition $p(0)=0$.
    \item 
    Let $\gamma=0$, and let $T$ be $\gamma$-homogeneous (vanilla, absolute, or positive). If $p(1)\neq 0$, then the only feasible eigenvalue (nonnegative eigenvalue for positive homogeneity) is $\lambda=0$. In addition, except for vanilla $\gamma$-homogeneity with $|F|=2$, $p(0)=0$.
\end{enumerate}
\end{claim}
\begin{proof}
\begin{enumerate}
\item 
Let $p(x)=x^m$ be the minimal polynomial of $T$. As already argued, every $0\neq v\in T^{m-1}(V)$ is an eigenvector with $\lambda=0$. If there were $v\neq 0$ with some eigenvalue $0\neq\lambda\in F$, then it is easily shown by induction that $T^m(v)=\lambda^{\sum_{i=1}^{m-1} \gamma^i}\lambda v$ for a $\gamma$-homogeneous $T$ or $T^m(v)=|\lambda|^{\sum_{i=1}^{m-1} \gamma^i}\lambda v$ for an absolutely $\gamma$-homogeneous $T$. In both cases, since $T^m(v)=0$, we get $\lambda=0$ and a contradiction.
\item
Consider first a $\gamma$-homogeneous $T$ where $\gamma\in\NN$. Since $T$ is not nilpotent, $p=\sum_{i=0}^ma_ix^i$ has some element $a_j\neq 0$ where $j\neq m$. For every $0\neq a\in F$ and $v\in V$, it is easy to see that $T^i(av)=a^{\gamma^i} T^i(v)$, so $p(T(av))=\sum_{i=0}^m a_ia^{\gamma^i}T^i(v)$. We may divide by $a^{\gamma^m}$ and obtain the vanishing polynomial $\sum_{i=0}^m a_ia^{\gamma^i-\gamma^m}T^i(v)$. This polynomial is monic and has the same degree as $p$, so by the uniqueness of the minimal polynomial, $a^{\gamma^j-\gamma^m}=1$. If $\gamma\notin\{0,1\}$, then $a^{\gamma^m-\gamma^j}=1$ has at most $\gamma^m-\gamma^j$ solutions, so in an infinite $F$ we may choose $a\neq 0$ that is not such a solution. This leads to a contradiction.

For positive and absolute $\gamma$-homogeneity, we can choose $a=2$ and use the same argument except that now $a^{\gamma^m-\gamma^j}=1$ cannot not hold by simple observation.
\item 

Let $v\neq 0$ be an eigenvector with eigenvalue $\lambda$ and let $T$ be $1$-homogeneous. For $\lambda\neq 0$, $0 = p(T)(v) = p(\lambda)v$ so $p(\lambda)=0$. For $\lambda=0$, note that there exists $a\in F\setminus \{0,1\}$, so $T(0)=0$, and by part~2 of Claim~\ref{cla:0 1 eigenvalues}, again $p(\lambda)=0$.

If $T$ is absolutely $1$-homogeneous, then for every $0\neq a\in F$ and $v\in V$, 
\begin{align}
    0&=p(T)(av)=a_0av+\sum_{i=1}^ma_i|a|T^i(v)\\
    &=a_0(a-|a|)v+|a|p(T)(v)\\
    &=a_0(a-|a|)v\;.
\end{align}
Taking $v\neq 0$ and $a=-1$, we obtain $a_0=0$, so $p(0)=0$. Now, let $v\neq 0$ be an eigenvector with eigenvalue $\lambda$. If $\lambda=0$, then $p(|\lambda|)=0$, and if $\lambda\neq 0$, then
$0 = p(T)(v) = (\sum_{i=1}^m a_i|\lambda|^{i-1}\lambda)v = \sgn(\lambda)p(|\lambda|)v$, so $p(|\lambda|)=0$, and we are done.

\item 
For all $\gamma$-homogeneity types, it holds that for every $v\in V$ and permissible $0\neq a\in F$,
\begin{align}
    0&=p(T)(av)=a_0av+\sum_{i=1}^m a_iT^i(v)\\
    &=a_0(a-1)v+p(T)(v)\\
    &=a_0(a-1)v\;.
\end{align}
If $v\neq 0$ and $|F|>2$, there is a permissible $a\notin\{0,1\}$, yielding $a_0=0$, so $p(0)=0$. 

If $v$ has eigenvalue $\lambda\neq 0$ ($\lambda>0$ for positive homogeneity), then $0=p(T)(v)=\sum_{i=1}^m a_i\lambda v +a_0v=p(1)\lambda v+a_0(1-\lambda)v=p(1)\lambda v$. The last equality holds since either $a_0=0$ or $|F|=2$, so $\lambda=1$. Thus, $p(1)=0$, and the result follows. 
\end{enumerate}
\end{proof}

Thus, for (absolutely) $1$-homogeneous operators, if a vanishing polynomial exists, its roots must contain the eigenvalues (their absolute value), under proper conditions. However, not all the roots must be eigenvalues. Consider, for example, $T:\RR^2\rightarrow\RR^2$, defined (using polar coordinates) as 
\begin{align}
    T((r,\theta)) &= \begin{cases}
      (r,\theta+\frac{\pi}{2}), &\hspace{-4pt} \theta\in[0,\frac{\pi}{2})\cup [\pi,\frac{3\pi}{2}) \\
      (r,\theta-\frac{\pi}{2}), &\hspace{-4pt} \text{otherwise.}
    \end{cases}
\end{align}
This operator is $1$-homogeneous and has a vanishing (and minimal) polynomial $x^2-1$, whose roots are $\pm1$. Nevertheless, it has no eigenvectors at all. Another example is $T(v)=|v|$ with $V=\RR$, which is absolutely $1$-homogeneous with a minimal polynomial $x^2-x$, whose roots are $\{0,1\}$, but only the eigenvalue $1$ has an eigenvector.

An example of a positive $0$-homogeneous operator is $T(v)=v/\|v\|$ (where $T(0)=0$) in $\RR^n$. Every $v\neq 0$ is an eigenvector with $\lambda=1/\|v\|$, yielding every eigenvalue in $(0,\infty)$. The minimal polynomial of $T$ is $x^2-x$, so $p(1)=0$. By part 4 of Claim~\ref{cla:homgeneous eigenvalues} we cannot limit the eigenvalues, and indeed, they attain infinitely many values.

 \section{Generalized Inverses of Nonlinear Endofunctions}
 It is natural to consider the Drazin inverse as a generalized inverse for the setting of endofunctions over any set $V$. Of course, $\{1,2\}$-inverses and the MP inverse may be applied to the special case $V=W$ exactly as they are in general. 
 
 The next theorem gives a necessary and sufficient condition for the existence of the Drazin inverse along with an exact formula. We will need a small lemma first.
 \begin{lemma}\label{lem:weak drazin injective}
 Let $T\in OP(V)$ satisfy that for some integer $k\geq 0$, there exists $G\in OP(V)$ s.t. $GT^{k+1}=T^k$. Then $S = T|_{T^k(V)}$ is injective. 
 \end{lemma}
 \begin{proof}
 Suppose that $v_1,v_2\in T^k(V)$ and $S(v_1)=S(v_2)=v$. There are $u_1,u_2\in V$ s.t. $v_1=T^k(u_1)$ and $v_2=T^k(u_2)$, and thus $T^{k+1}(u_1)=T^{k+1}(u_2)=v$. It holds that 
$G(v)=GT^{k+1}(u_1)=T^k(u_1)=v_1$ and $G(v)=GT^{k+1}(u_2)=T^k(u_2)=v_2$, so $v_1=v_2$.
 \end{proof}

  \begin{theorem}\label{thm:drazin characterization main}
 If $T\in OP(V)$, then $T$ has a Drazin inverse iff there is $k\in\NN$ s.t. $S:T^k(V)\rightarrow T^k(V)$, defined as $S=T|_{T^k(V)}$, is bijective. If the condition holds, then $\drinv{T} = S^{-(k+1)}T^k$.
 \end{theorem}
 \begin{proof}
If $\drinv{T}$ exists, then for some positive $k$, $T^k=T^{k+1}\drinv{T}$. Thus, $T^k(V)=(T^{k+1}\drinv{T})(V)\subseteq T^{k+1}(V)$, but since we always have $T^k(V) \supseteq T^{k+1}(V)$, it holds that $T^k(V)=T^{k+1}(V)$. Thus $S$ is surjective. Now, since $\drinv{T}T^{k+1}=T^k$, then by Lemma~\ref{lem:weak drazin injective}, $S$ is injective, and being also surjective, it is a bijection.

In the other direction, assume that for some $k\geq 0$, $S$ is bijective. Thus, $S$ has a two-sided inverse. We will define $G=S^{-(k+1)}T^k$ and show that it is a Drazin inverse. We have $GT=S^{-(k+1)}T^{k+1}=S^{-k}T^k=TG$, so D5 holds. MP1$^k$ is true since $T^kGT=T^{k+1}G=T^k$, implying MP1$^{k+1}$. MP2 holds since $GTG=T^{k+1}S^{-2k-2}T^k=S^{-(k+1)}T^k=G$, so $G$ is a $\{1^{k+1},2,5\}$-inverse, and the proof is complete.
 \end{proof}
 The exact setting for the existence of a Drazin inverse is shown in Figure~\ref{fig:drazinA}. While it holds for any operator that $V=T^0(V)\supseteq T^1(V) \supseteq \ldots$, Drazin-invertible operators are those for which the containment becomes an equality at some point, and furthermore, $T$ becomes bijective (not only surjective).
\ignore{
}

\begin{figure*}[ht]
	\centering
	\begin{subfigure}{0.5\textwidth}
		\includegraphics[scale=0.35,trim={-3cm, 0cm, 12cm, 0cm},clip]{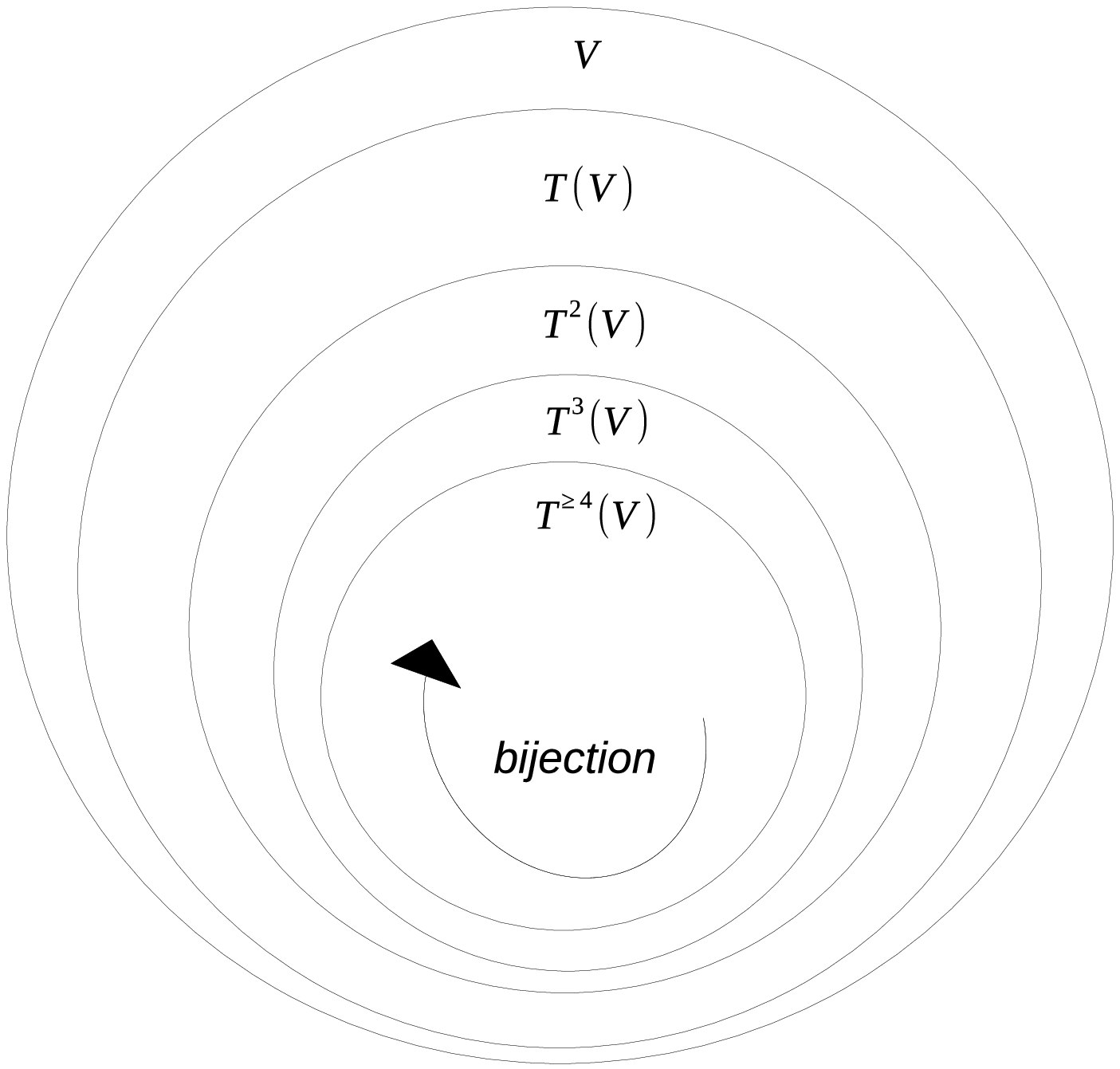}
		\caption{A Drazin-invertible operator}
		\label{fig:drazinA}
	\end{subfigure}
	\begin{subfigure}{0.45\linewidth}
		\includegraphics[scale=0.35,trim={-2cm, 0cm, 12cm, 0cm},clip]{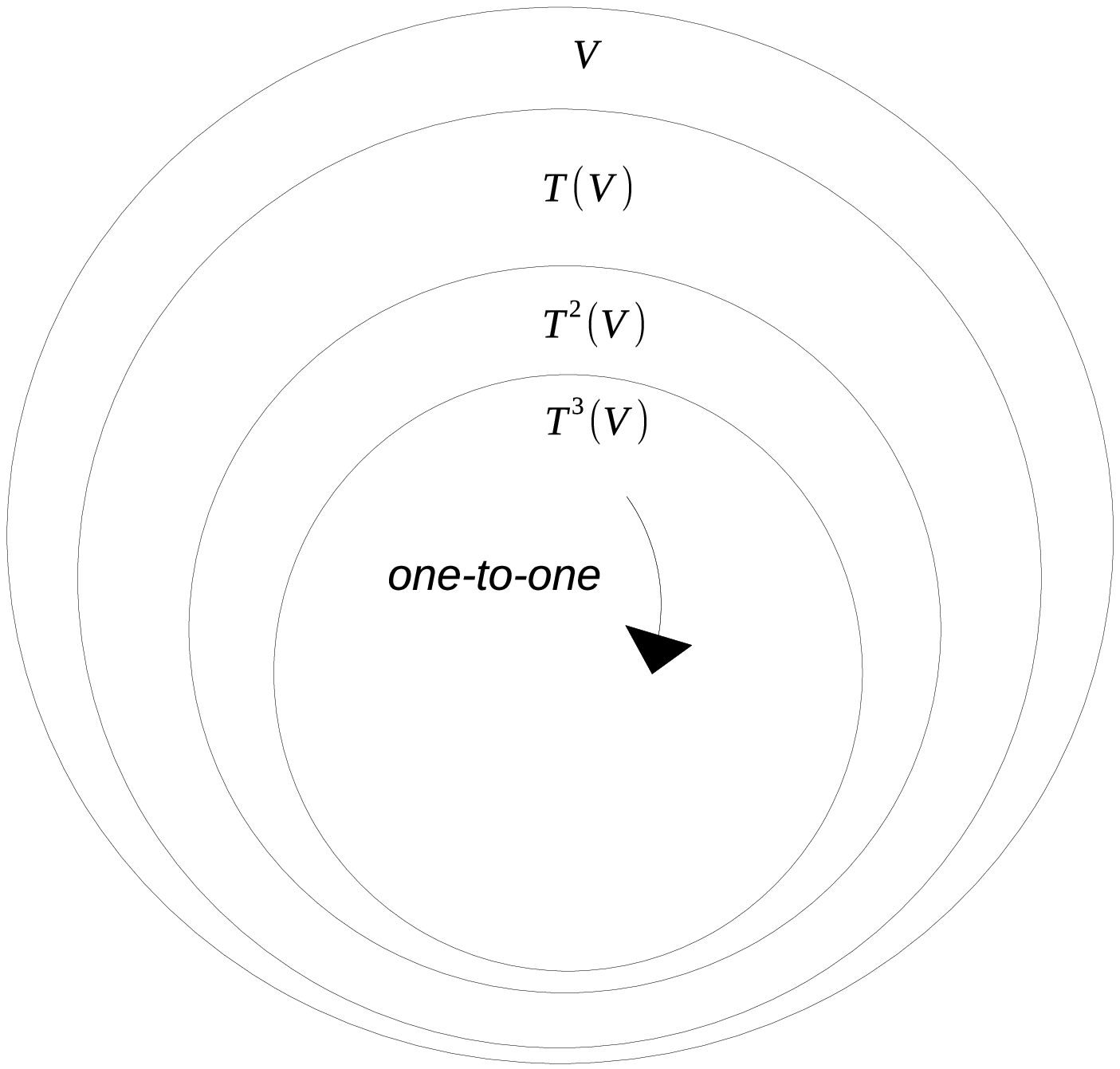}
		\caption{A left-Drazin-invertible operator}
		\label{fig:drazinB}
	\end{subfigure}	
	\caption{Necessary and sufficient settings for Drazin-invertibility and left-Drazin invertibility. For a Drazin-invertible operator, the operator $T$ becomes bijective after a finite number of iterations (here, $k=4$).  The Drazin inverse is then $\drinv{T}=(T|_{T^k(V)})^{-(k+1)}T^k$. For a left-Drazin-invertible operator, the setting is looser: the operator becomes \textit{injective} after $k$ iterations (here, $k=3$). An inverse $\drlinv{T}$ with parameter $k$ (for $k\geq 1$) may be then defined as $(T|_{T^k(V)})^{-1}$  on $T^{k+1}(V)$ and arbitrarily elsewhere.}
	\label{fig:drazin and left drazin}
\end{figure*}
 
\subsection{Some Example Scenarios}
The following scenarios serve to further investigate and highlight the concept and properties of the Drazin inverse for operators.

\textbf{An operator containing a loop.} Let $T\in OP(V)$ satisfy $T^n=T^k$, where $V$ is some set and $k,n\in \NN$ satisfy $n>k\geq 0$. This holds for any operator on a finite $V$ (see Claim~\ref{cla: gen for finite num vals}) and all examples mentioned in Subsection~\ref{subsubsec:examples by degree}, such as idempotent operators and nested projections. 

All these operators have $\drinv{T}=T^{(n-k)(k+1)-1}$. To see this, we can use part~8 of Theorem~\ref{thm:drazin}, with $a=b=T^{n-k-1}$ and $M=p=q=k$. Assuming that $k>0$, we obtain $\drinv{T}=T^kT^{(n-k-1)(k+1)}=T^{(n-k)(k+1)-1}$. For $k=0$, $T^{n-1}$ is clearly the true inverse of $T$, and thus also the Drazin inverse.
 
\textbf{The bijection of $\boldsymbol{T^k(V)}$ is a unitary matrix.} If $V=\RR^n$, $T^k(V)=B(0,r)$, and $T|_{T^k(V)}$ is a unitary matrix, then Theorem~\ref{thm:drazin characterization main} implies that $T$ is Drazin-invertible. This unitary matrix has a vanishing polynomial $p$ of degree at most $n$, which can be non-trivial (even a rotation in $\RR^2$ does not generally have a vanishing polynomial of the form $x^m-x^k=0$). The polynomial $p(x)x^k$ vanishes in $T$.

\textbf{The bijection of $\boldsymbol{T^k(V)}$ is arbitrary.}
Using the previous example except that the bijection is completely arbitrary, we cannot expect that the bijection will vanish for some polynomial. Drazin-invertibility may thus exist without necessarily having a vanishing polynomial.
\ignore{}

\textbf{A simple operator without a Drazin inverse.}
Let $V=[0,1]$, and let $T:V\rightarrow V$ be defined as $T(v)=v/2$. Since $T^k(V)\supsetneq T^{k+1}(V)$ for every $k$, then by Theorem~\ref{thm:drazin characterization main}, a Drazin inverse does not exist. In contrast, the pseudo-inverse of $T$ with the Euclidean norm is easily seen to be $\min\{2v,1\}$. Note that this linear operator has the vanishing polynomial $x-1/2=0$.

\subsection{The Left-Drazin Inverse}
The examples given in the previous subsection showed that an operator with a vanishing polynomial does not necessarily have a Drazin inverse. For some very simple cases, such as $T^n=T^k$, a Drazin inverse can be shown to be polynomial in the operator. In general, however, this cannot be expected. The D5 requirement means that an operator and its Drazin inverse should commute. While $T$ and a polynomial in $T$ commute for linear operators, they do not commute in general. 

This is a bit counter-intuitive, since for an operator with a vanishing polynomial $q(x)x^{k+1}+a_kx^k$, where $a_k\neq 0$, the expression  $-a_k^{-1}q(T)$ is ``inverse-like''. Specifically, $-a_k^{-1}q(T)\circ T^{k+1}=T^k$. Together with the fact that the Drazin inverse was shown to be undefined for an elementary case, this motivates a possibly weaker but more permissive generalized inverse definition, given next.
 
 \begin{definition}[Left-Drazin inverse]
 An operator $T\in OP(V)$ has a left-Drazin inverse $G\in OP(V)$ if there is a positive integer $m$ s.t. $GT^{m+1}=T^m$. The minimal such integer is referred to as the left index of $T$.
 \end{definition}
 Having a left-Drazin inverse is referred to by Drazin as ``left $\pi$-regularity'' and is in turn inspired by Azumaya \citep{azumaya1954strongly}. Clearly, the Drazin inverse of $T$, where it exists, is also a left-Drazin inverse of $T$ (see Theorem~\ref{thm:drazin} part~8). In addition, if a parameter $m$ satisfies the definition then so does any $m'\geq m$.
 
 The next theorem characterizes the left-Drazin inverse in the context of endofunctions and shows that it exists as a polynomial in the operator whenever a vanishing polynomial exists.
 
 \begin{theorem}\label{thm:left drazin}
 Let $T\in OP(V)$.
 \begin{enumerate}
 \item
  It holds that $T$ has a left-Drazin inverse $\drlinv{T}$ iff there is $k\in \NN$ s.t. $S:T^k(V)\rightarrow T^{k+1}(V)$, defined as $S=T|_{T^k(V)}$, is injective. If the condition holds and $k\geq 1$, then a left-Drazin inverse with parameter $k$ exists, must coincide with $S^{-1}$ on $T^{k+1}(V)$, and may be arbitrary elsewhere.
  \item 
  Let $V$ be further a vector space and $p=\sum_{i=0}^m a_ix^i$ a vanishing polynomial of $T$, where $k$ is smallest index s.t. $a_k\neq 0$. Then $-a_k^{-1}\sum_{i=k+1}^m a_iT^{i-k-1}$ is a left-Drazin inverse of $T$ with parameter $\max\{k,1\}$.
  \end{enumerate}
 \end{theorem}
\begin{proof}
  \begin{enumerate}
  \item
  The existence of $\drlinv{T}$ implies that $S$ is injective (for some $k$) immediately by Lemma~\ref{lem:weak drazin injective}. In the other direction, we may define $\drlinv{T}$ as $S^{-1}$ on $T^{k+1}(V)$ and arbitrarily elsewhere. It then holds that $\drlinv{T}T^{k+1} = T^k$ as required (for $k=0$, note that also $\drlinv{T}T^{k+2} = T^{k+1}$). The last claim is immediate.
  \item 
  Since $T^k = \left(-a_k^{-1}\sum_{i=k+1}^m a_iT^{i-k-1}\right)T^{k+1}$, the claim follows for $k>0$. For $k=0$ we may multiply both sides by $T$ on the right.\qedhere
  \end{enumerate}
\end{proof}
The characterization of left-Drazin invertibility is depicted in Figure~\ref{fig:drazinB}. Part~1 of the theorem directly yields the following.
\begin{corollary}
If $T\in OP(V)$ is bijective, then $\drlinv{T}=T^{-1}$.
\end{corollary}

Going back to the example of $T(v)=v/2$ on $V=[0,1]$, we have that $T$ is injective on $T^0(V)$, so a left-Drazin inverse exists, which with parameter $1$ is constrained to be $2v$ on $[0,0.25]$. The same conclusion might be reached by a minor adaptation of the proof of part~2 of the above theorem using the vanishing polynomial $p(x)=x-0.5$. 

Note that again by the above theorem, the operator $T(v)=2v\cdot\Ind{v<1/2}+v\cdot\Ind{v\geq 1/2}$ with $V=[0,1]$ does not have a left-Drazin inverse, since it is not injective on $T^k(V)$ for any $k$. The pseudo-inverse of $T$ with the Euclidean norm exists and is $v/2$.

\section{Discussion and Conclusion}
Inversion and pseudo-inversion are essential in machine learning, image processing and data science in general.
In this work we attempt to establish a broad and coherent theory for the generalized inverse of nonlinear operators. From an axiomatic perspective, the first two axioms (MP1, MP2) can be directly extended to a very general setting of the nonlinear case, yielding a set of admissible $\{1,2\}$-inverse operators. It is shown that (MP3, MP4) are requirements which can hold only in the linear case. A procedure to construct a $\{1,2\}$-inverse is given, along with some essential properties. In order to obtain a stricter definition, which fully coincides in the linear setting with the four axioms MP1--4, normed spaces are assumed. The pseudo-inverse is defined through a minimization problem, as well as MP2.
We note that such a minimization can be thought of as the limit of Tikhonov regularization, where the weight of the regularization term tends to zero. This analogy applies in particular to regularized loss minimization (see Appendix~\ref{sec:rlm} for details). We give explicit formulations of the nonlinear pseudo-inverse for several test cases, relevant to learning and to image processing.

Finally, we focus on operators with equal domain and range, and in particular on expressing inverse and generalized inverse by using vanishing polynomials, following the Cayley-Hamilton theorem for linear operators. We show why the Drazin inverse is relevant in this case and suggest a relaxed version, more suitable for obtaining polynomial inverse expressions for nonlinear operators. The use of polynomials allows realizing inversion using only forward applications of the operator, yielding a constructive and possibly efficient way for inverse computation.

\ignore{}

\bmhead{Acknowledgments}
We acknowledge support by grant agreement No. 777826 (NoMADS), by the Israel Science Foundation (Grants No. 534/19 and 1472/23), by the Ministry of Science and Technology (Grant No. 5074/22) and by the Ollendorff Minerva Center.

% 
% 

%\newpage

\begin{appendices}

\section{Additional Claims}\label{sec:AddClaims}
\begin{claim}\label{cla:adjoint means linear}
Let $V$ and $W$ be inner-product spaces over $F$ ($\RR$ or $\CC$). If $T:V\rightarrow W$ has an adjoint operator, then $T$ is linear.
\end{claim}
\begin{proof}
For every $v_1,v_2\in V$, $w\in W$, and $a_1,a_2\in F$, it holds that 
\begin{align}
    \langle T(a_1 &v_1+a_2 v_2), w \rangle \\
    &= \langle a_1 v_1+a_2 v_2, T^*(w)\rangle\\
    &=  a_1 \langle v_1, T^*(w)\rangle + a_2 \langle v_2, T^*(w)\rangle \\
    &=  a_1 \langle T(v_1), w\rangle + a_2 \langle T(v_2), w\rangle \\
    &=\langle a_1T(v_1)+a_2 T(v_2), w \rangle\;.
\end{align}
\ignore{
}
Taking $w = T(a_1 v_1+a_2 v_2) - a_1T(v_1) - a_2 T(v_2)$ and rearranging, we obtain that $\langle w,w \rangle=0$, therefore $w=0$, yielding that $T$ has to be linear. 
\end{proof}
\begin{claim}\label{cla:full inv of a bijection}
If $T:V\rightarrow W$ is a bijection, then $T^{-1}$ is its unique pseudo-inverse.
\end{claim}
\begin{proof}
A pseudo-inverse is a $\{1,2\}$-inverse, and by Lemma~\ref{lem: equivalent meaning of MP1-2 different spaces}, $T^{-1}$ is the unique $\{1,2\}$-inverse of $T$. For every $w\in W$, $v=T^{-1}(w)$ is the sole minimizer of $\|T(v)-w\|$, so the BAS property is satisfied by $T^{-1}$ as well.
\end{proof}

The following lemma and theorem are well known. Their statements below extend the original ones to spaces over $\CC$ as well as $\RR$, so the proofs are given for completeness.
\begin{lemma}[\citep{cheney1959proximity}]\label{lem:cheny-goldstein lemma}
Let $C$ be a convex set in a Hilbert space over $\RR$ or $\CC$. If a point $b\in C$ is nearest a point $a$, then $\textup{Re}\langle x-b,b-a \rangle \geq 0$ for every $x\in C$.
\end{lemma}
\begin{proof}
For every $t\in [0,1]$, $tx+(1-t)b\in C$, and thus $\|a-b\|^2\leq \|a-tx-(1-t)b\|^2$. Now,
\ignore{

}
\begin{align}
\|&a-tx-(1-t)b\|^2 \\
&= \|t(b-x)+(a-b)\|^2 \\
&= \|t(x-b)+(b-a)\|^2 \\  
&= t^2\|x-b\|^2+2t\textup{Re}\langle x-b,b-a \rangle\\
&\hspace{10pt} + \|a-b\|^2\;,    
\end{align}
\ignore{

}
therefore, $t^2\|b-x\|^2+2t\textup{Re}\langle x-b,b-a \rangle\geq 0$. If $\textup{Re}\langle x-b,b-a \rangle < 0$, then this inequality would be violated for a small enough $t$, so we are done.
\end{proof}

\begin{theorem}[\citep{cheney1959proximity}]\label{thm:projection is continuous}
The projection $P$ onto a nonempty closed convex set $C$ in a Hilbert space over $\RR$ or $\CC$ satisfies the Lipschitz condition $\|P(x)-P(y)\|\leq \|x-y\|$.
\end{theorem}
\begin{proof}
By Lemma~\ref{lem:cheny-goldstein lemma}, $A\equiv\textup{Re}\langle P(x)-P(y),P(y)-y\rangle\geq 0$ and $B\equiv\textup{Re}\langle P(y)-P(x),P(x)-x\rangle\geq 0$. Rearranging terms in the inequality $A+B\geq 0$, we have
\begin{align}
     \textup{Re}\langle P(y)&-P(x),y-x\rangle \\&\geq \textup{Re}\langle P(y)-P(x),P(y)-P(x)\rangle \\
     &= \|P(y)-P(x)\|^2\;.
\end{align}
\ignore{

}
Then, by the Cauchy-Schwartz inequality,
\begin{align}
    \|P(y)&-P(x)\|\cdot\|y-x\|\\
    &\geq  |\langle P(y)-P(x),y-x\rangle|\\
    & \geq 
     \textup{Re}\langle P(y)-P(x),y-x\rangle \\
     &\geq  \|P(y)-P(x)\|^2\;.
\end{align}
\ignore{

}
Therefore, $\|P(y)-P(x)\| \leq \|y-x\|$.
\end{proof}

\section{Pseudo-Inverse Calculations for Table~\ref{tab:pinv1d}}\label{sec:1d pseudo inverse proofs}
In all the examples below, $T:V\rightarrow W$ and $V=W=\RR$.

Let $\bm{T(v)=v^2}$. For $w>0$ there are two sources, $v=\pm\sqrt{w}$, with the same norm. For $w<0$ there is no source, and the closest element in the image is $0$, whose only source is $0$. By BAS, a pseudo-inverse $\finv{T}$ must satisfy $\finv{T}(w) =\pm\sqrt{\max\{w,0\}}$, where the sign for each $w$ is arbitrary. Each such choice produces a $\{1,2\}$-inverse by Lemma~\ref{lem: equivalent meaning of MP1-2 different spaces}, since $\finv{T}(w)$ is a source of $w$ for $w\in T(V)$, and a pseudo-inverse of an element in $T(V)$, otherwise.

Let $\bm{T(v)=v\cdot\Ind{|v|\geq a}}$, $a\geq 0$ (hard thresholding). Assume that $a>0$. By BAS, a pseudo-inverse $\finv{T}$ must satisfy $\finv{T}(w) = w$ if $|w|\geq a$, $\finv{T}(w) = a$ if $\frac{a}{2}<w<a$, $\finv{T}(w) = -a$ if $-a<w<-\frac{a}{2}$, and $\finv{T}(w)=0$ if $|w|\leq \frac{a}{2}$. More compactly, $\finv{T}(w)=\sgn(w)\cdot\Ind{|w|>\frac{a}{2}}\max(a,|w|\}$. MP2 may be verified directly, so $\finv{T}$ is valid and also unique. If $a=0$, then $T(v)=v$, and the unique pseudo-inverse of this bijection is clearly $\finv{T}(w)=w$, which agrees with the pseudo-inverse expression given for the case $a>0$. 

Let $\bm{T(v)=\sgn(v)\max\{|v|-a,0\}}$, $a\geq 0$ (soft thresholding). By BAS, any candidate pseudo-inverse $\finv{T}$ must have $\finv{T}(w) = w + a$ for $w>0$, $\finv{T}(w) = w - a$ for $w<0$ and $\finv{T}(0) = 0$. More compactly, $\finv{T}(w) = \sgn(w)(|w|+a)$. MP2 is easily verified directly, so our unique candidate is also valid.

Let $\bm{T(v)=\tanh(v)}$. The image of $T$ is $(-1,1)$, and as a result, for $w\notin (-1,1)$ there is no nearest element in the image, and the pseudo-inverse cannot be defined for such $w$. If we rather choose $W=(-1,1)$, then $T$ is a bijection, and therefore $\finv{T} = T^{-1} = \arctanh$.

Let $\bm{T(v)=\sgn(v)}$. For $w< -\frac{1}{2}$, the nearest value in $T(V)$ is $-1$, but there is no source with minimal norm. A similar situation holds for $w>\frac{1}{2}$. For $w=\pm\frac{1}{2}$ there are two nearest elements in $T(V)$, and the source with minimal norm is $0$. For $w\in (-\frac{1}{2},\frac{1}{2}) $ the nearest element in $T(V)$ is $0$ and its single source is $0$. To define a pseudo-inverse $\finv{T}$ we must restrict its domain to $[-\frac{1}{2},\frac{1}{2}]$. We have seen that to satisfy BAS, necessarily $\finv{T}(w)=0$. That is a $\{1,2\}$-inverse by Lemma~\ref{lem: equivalent meaning of MP1-2 different spaces}, since for $w=0$, $\finv{T}(0)$ is a source, and for $w\neq 0$, $\finv{T}(w) = \finv{T}(0)$.

Let $\bm{T(v)={\sgn}_\epsilon(v)}$, $0<\epsilon\in \RR$. Recall that $\sgn_\epsilon(v)=\min\{1,\max\{-1,v/\epsilon\}\}$. Here, $T(V)=[-1,1]$, and every $w\in (-1,1)$ has the single source, $\epsilon w$. For $w=1$, the source with minimal norm is $\epsilon$, and for $w=-1$, the source with minimal norm is $-\epsilon$. In addition, for every $w$ there is a single nearest point in $T(V)$. By Claim~\ref{cla:exist and unique}, there exists a unique pseudo-inverse, $\finv{T}:\RR\rightarrow\RR$. For $w\in [-1,1]$, it must satisfy $\finv{T}(w)=\epsilon w$, as already seen. For $w>1$, BAS necessitates that $\finv{T}(w) = \epsilon$, and similarly for $w<-1$, that $\finv{T}(w) = -\epsilon$. Thus,
\begin{align}
\finv{T}(w)&= \epsilon\min\{1,\max\{-1,w\}\} \\&= \epsilon\sgn_1(w)\;,
\end{align}
and this pseudo-inverse is valid and unique.

Let $\bm{T(v)=\exp(v)}$. The image of $T$ is $(0,\infty)$, and for $w\leq 0$ there is no nearest element in the image, and $\finv{T}(w)$ cannot be defined. Since $T$ is bijective from $\RR$ to $(0,\infty)$, for $w\in(0,\infty)$, $\finv{T}(w)=\log (w)$.

Let $\bm{T(v)=\sin(v)}$. The image of $T$ is $[-1,1]$. Every element in the image has infinitely many sources, but the one with the smallest norm is $\arcsin(w)\in [-\frac{\pi}{2},\frac{\pi}{2}]$. Every $w$ has a single nearest point in $[-1,1]$, so by Claim~\ref{cla:exist and unique}, there exists a unique pseudo-inverse $\finv{T}:\RR\rightarrow \RR$. By BAS, it has to be $\arcsin(1)$ for $w>1$ and $\arcsin(-1)$ for $w<-1$, and in summary, 
$\finv{T}(w)=\arcsin(\min\{1,\max\{-1,w\}\})$.

%%%=============== END APP B ================

\section{Relation to Regularized Loss Minimization}\label{sec:rlm}
There is an interesting connection between the BAS property and Tikhonov regularization. For an operator $T:\RR^k\rightarrow\RR^m$, given $w\in\RR^m$, BAS seeks $v\in\RR^k$ that minimizes $\|Tv-w\|$ and among such minimizers, then seeks to minimize $\|v\|$. As noted by \citet{ben2003generalized} in the context of linear operators, Tikhonov regularization seeks a different, yet related, goal of minimizing $\|Tv-w\|^2+\lambda\|v\|^2$, where $\lambda\in\RR^+$. Namely, the two-stage minimization process is replaced by a single step combining both, which in this case describes ridge regression.

Tikhonov regularization for nonlinear operators is central to modern machine learning, as part of the framework of regularized loss minimization (RLM). Briefly, in a supervised learning task, a learner is given a training set $S = \{(x_i,y_i)\}_{i=1}^m$, typically with $x_i\in\mathcal{X}=\RR^n$ and $y_i\in\mathcal{Y}=\RR$ or $\{-1,1\}$. The learner outputs a predictor $p(\cdot|v):\mathcal{X}\rightarrow\mathcal{Y}$ that is parametrized by a vector $v\in\RR^k$ and is identified with it. The goal of the learner is to minimize the loss, $\ell(v,(x,y))$, on unseen test examples $(x,y)$, where $\ell$ is a given non-negative function. For various reasons (see, e.g., \citep{shalev2014understanding}), RLM may obtain this goal by minimizing $L_S(v) + R(v)$, where
$L_S(v) = \frac{1}{m}\sum_{i=1}^m \ell(v,(x_i,y_i))$ is the training loss and $R:\RR^k\rightarrow \RR$ is a regularization term. 

Consider then the operator $T:\RR^k\rightarrow\RR^m$, defined by $T(v)=(p(x_1|v),\ldots,p(x_m|v))$. Let both spaces be equipped with the Euclidean norm, and let $w = (y_1,\ldots,y_m)$. Then $v=\finv{T}(w)$ should minimize $\|T(v)-w\|$, which is equivalent to minimizing $\sum_{i=1}^m (p(x_i|v) - y_i)^2$, or $mL_S$, where $\ell(v,(x,y))=(p(x_i|v) - y_i)^2$ is the squared loss. Among minimizers, one with minimal norm is chosen. This may be approximated by minimizing $\sum_{i=1}^m (p(x_i|v) - y_i)^2 + m\lambda\|v\|^2$, where $m\lambda$ is positive and sufficiently small. That way, the loss term is dominant, and the regularization term is effective only as a near-tie breaker. Of course, this is implied by the BAS property alone, and a pseudo-inverse should also satisfy MP2.

%--------------------- END APPENDIX

\end{appendices}


\end{document}